\documentclass[12pt,reqno]{amsart}
\usepackage{tikz,dsfont}
\usepackage{amsmath, amssymb, graphicx, mathrsfs, hyperref}
\usepackage{pdfpages}
\usepackage{verbatim} 
 \usepackage{float}
\usepackage{enumerate}
\newtheoremstyle{dotless}{}{}{\itshape}{}{\bfseries}{}{ }{} 
\theoremstyle{dotless}

\newtheorem{theorem}{Theorem} 

\newtheorem{lemma}{Lemma}

\newtheorem{prop}{Proposition}

\newtheorem{corollary}{Corollary}

\newtheorem{remark}{Remark}

\usepackage{color}
\definecolor{red}{rgb}{1,0,0}
\definecolor{blue}{rgb}{.2,.6,.75}
\definecolor{green}{rgb}{.4,.7,.4}

\title[Mixed incomplete character sums   with smooth moduli]{Mixed incomplete character sums of rational functions with smooth moduli}
\author{Todd Cochrane}
\address{ Dept of Mathematics\\
          Kansas State University\\
          Manhattan, KS 66506}
\email{cochrane@math.ksu.edu} 
\author{Andrew Granville}
\address{D{\'e}partment  de Math{\'e}matiques et Statistique,   Universit{\'e} de
Montr{\'e}al, CP 6128 succ Centre-Ville, Montr{\'e}al, QC  H3C 3J7, Canada.}
\email{and.granville@gmail.com}
\author{Junren Zheng}
\address{Junren Zheng,
School of Mathematics and Statistics, Xi'an Jiaotong University, Xi'an, China, Xi'an, 710049, P. R. China and 
D{\'e}partment  de Math{\'e}matiques et Statistique,   Universit{\'e} de
Montr{\'e}al, CP 6128 succ Centre-Ville, Montr{\'e}al, QC  H3C 3J7, Canada.}
\email{junrenzheng03@gmail.com}
\thanks{AG is partially supported by a grant from the the Natural Sciences and Engineering Research Council of Canada.
   JZ is funded in part by the China Scholarship Council, NSFC (No. 12025106) and Shaanxi NSF (No. 2025JC-QYCX-002).}
\dedicatory{Dedicated to Roger Heath-Brown on the occasion of his 75th birthday}
\begin{document}

\begin{abstract}
Let $\chi=\chi_q$ be a primitive character mod $q$ and fix $\Delta>0$.
Graham and Ringrose \cite[Theorem 5]{GR} gave strong bounds on character sums $\sum_{M<n\leq M+N} \chi(n)$ in   intervals of length $N=q^\Delta$ whenever $q$ is squarefree and is sufficiently smooth.  Here we show that the smoothness parameter can be taken to be  $N^{1-\epsilon}$. We also discuss various generalizations and applications, obtaining best possible results in several aspects.
\end{abstract}

\maketitle

\section{Introduction}
Heath-Brown \cite{HB} developed a $q$-analogue of the  van der Corput method to estimate character sums, which enabled him to derive a Weyl-type bound for Dirichlet $L$-functions.
By employing this method, Graham and Ringrose \cite[Theorem 5]{GR} proved that 
for any fixed $\Delta>0$ there exists a $\xi=\xi(\Delta)>0$ such that if $\chi$ is a primitive character mod $q$ where $q$ is squarefree and $q^\xi$-smooth then there exists a constant $\eta=\eta(\Delta)>0$ for which
 \[
\bigg|  \sum_{M<n\leq M+N} \chi(n) \bigg| \ll N^{1-\eta}
 \]
for any integer $N$ with $N\geq q^\Delta$. 
In their proof, the smoothness bound $\xi$ is significantly smaller than the length of the interval parameter, $\Delta$; that is, $\xi(\Delta)\ll\Delta^2$ is admissible from their proof but  $\xi(\Delta)$ linear in  $\Delta$ is not.

\subsection{New results on incomplete character sums}
We will show that one may   take $\xi$ to be arbitrarily close to $\Delta$ in the Graham-Ringrose result, indeed that $q$ can be $N^{1-\epsilon}$-smooth. We remove the restriction that $q$ is squarefree replacing it with  $q\in \mathcal N(y)$, the set of positive  integers $q$ with  at most one prime factor $p\in (y, y^2]$ and otherwise all  prime power divisors of $q$ are $\leq y$, where $N\geq y^{1+\epsilon}$.  Our proof  will be adapted to establish a broad generalization, strong bounds on 
  \[
 \sum_{M<n\leq M+N} \chi(f(n)) e\bigg( \frac{ g(n)}q  \bigg),  
  \]
  where $\chi$ is a character mod $q$, and $f(x)$ and $g(x)$ are rational functions. 

We define $r_f$ to be the largest positive integer $r$ for which 
\begin{equation} \label{eq: Condition}
f(x)= cF(x)^r \text{ for some } F(x)\in \mathbb Q(x) \text{ and constant }  c\ne 0.
\end{equation}
 If $f(x)$ is a constant then define $\chi^{r_f}$ to be principal mod $q$ for any character $\chi \pmod q$.

\begin{theorem} \label{thm: main}  Fix $\delta,\epsilon>0$. Suppose that $f(x), g(x)\in \mathbb Q(x)$ where $f(x)$ and $g(x)$ are not both constants.
There exists a constant $\eta=\eta(f,g,\delta,\epsilon)>0$ and an explicitly determinable  positive integer $\mathcal M=\mathcal M(f,g,\delta)$ such that 
\begin{enumerate}[{\rm (a)}]
\item For any character $\chi$ mod $q$ where $q\in \mathcal N(y)$ with $(q,\mathcal M)=1$ and $y:=q^\delta$;
\item For any interval $I$ of length $N$ where $q\geq N\geq y^{1+\epsilon}$;
\end{enumerate}
we  have 
\begin{equation} \label{Key result}
  \sum_{n\in I} \chi(f(n)) e\bigg( \frac{ g(n)}q  \bigg) \ll N / q^{\eta},
 \end{equation} 
unless $g(x)$ is  a polynomial of degree $<1/\delta$  and $\chi^{r_f}$ is induced from a primitive character of conductor $q'$
 in which case we obtain
  \begin{equation} \label{Key result2}
  \sum_{n\in I} \chi(f(n)) e\bigg( \frac{ g(n)}q  \bigg) \ll N / (q')^{\eta}.
 \end{equation} 
  The implicit constants in \eqref{Key result} and \eqref{Key result2} depend only on $f, g, \delta$ and $\epsilon$  (like $\eta$).
\end{theorem}

A positive proportion of integers $q\leq x$ satisfy the hypothesis of   Theorem   \ref{thm: main}.

We claim that \eqref{Key result2} is ``best possible'': Let $g(x)=x^D$ for any integer $D<1/\delta$ and write 
$D= 1/\delta-\theta$  for some fixed $\theta>0$.  Let $I=(0,N]$ with $N=y^{1+\epsilon}$. Suppose
$\chi^{r_f}$ is induced from a primitive character $\psi$ of conductor $q'$ with $q'=o(q^{\epsilon})$ and that $q=rq'$ where $r$ is a large prime so that  from \eqref{eq: Condition} with $c=1$ we have
$ \chi(f(n))=  \psi(F(n)) 1_{(F(n),r)=1}$.
Then for any $n\in I$ we have $n^D\leq N^{1/\delta-\theta}= q^{(1-\delta\theta)(1+\epsilon)}<q^{1-\epsilon}$ for an appropriately small choice of $\epsilon>0$ so that $e(\frac{n^D }q)=1+O(q^{-\epsilon})=1+o(1/q')$ which gives
\[
\sum_{n\in I} \chi(f(n)) e\bigg( \frac{ g(n)}q  \bigg) =  \sum_{n\in I}   \psi(F(n)) +o(N/q')\sim  \frac N{q'}  \sum_{a \pmod {q'}}   \psi(F(a))
 \]
assuming that this sum is $\geq 1$ in absolute value (typically we expect the sum to be of size $(q')^{1/2+o(1)}$).
Therefore one cannot improve the upper bound in \eqref{Key result2}, at this level of generality, other than in the value of the exponent $\eta$.
    
There are similar but weaker results in \cite[Sections 12.5 \& 12.6]{IK}. In \cite[Section 12.5]{IK} the authors work with $q$ squarefree, take $g(x)=0$, and work with stronger hypotheses. In \cite[Section 12.6]{IK} the authors work with more powerful moduli but then with $f(x)=x, g(x)=0$ and only get good bounds when $M\leq N^{1+\eta}$ where $I=(M,M+N]$.


 We can let $\delta\to 0$ as $q\to \infty$ in  Theorem \ref{thm: main}. Our proof, as currently developed, needs $\delta \gg \frac 1{\log \log q}$ (see remark \ref{rem: MCC}) which is more-or-less the same range of uniformity obtained by
Mei-Chu Chang in \cite{mcc}  (though she had $f(x)=x, g(x)=0$ and more restrictive conditions on $q$). She gets the wider  range  $N>\exp( (\log q)^{1- c})$ when the squarefree part of $q$ is $\ll \exp( c'(\log q)^{1- c})$.

 In our proofs we will not need to assume that $(q,\mathcal M)=1$ but rather obtain bounds like \eqref{Key result} and \eqref{Key result2} in terms of $q_{\mathcal M}$ and $q'_{\mathcal M}$ where  $a_{\mathcal M}$ is defined to be the largest divisor of $a$ that is coprime with  a given integer $\mathcal M$.  
 We suspect that the integer $\mathcal M$ could be avoided in the formulation of Theorem \ref{thm: main}. It is used in the current proof to  avoid certain algebraic issues, and these may become more complicated for the exponential sums mod $p^m$ when $p|\mathcal M$; nonetheless these issues might be worth exploring to make our results more applicable, starting with the simplest case, $f(x)=x, g(x)=0$.


\subsection{Linear $g(x)$  and the Brun-Titchmarsh theorem}    
 
   Taking $g(x)=bx$ in Theorem \ref{thm: main} (and changing $q'$ to $Q$) we obtain the following, which can be viewed as bounding certain Fourier coefficients non-trivially:
 
 \begin{corollary}\label{cor: mainnonprincipal}
 Let $f(x)\in \mathbb Q(x)\setminus \mathbb Q$.
  Fix $\delta,\epsilon>0$. 
There exists a constant $\eta=\eta(\delta,\epsilon,f)>0$ and an explicitly determinable  positive integer $\mathcal M=\mathcal M(\delta,f)$ such that if $q\in \mathcal N(y)$  where $y:=q^\delta$ with $(q,\mathcal M)=1$ then for any interval $I$ of length $N$ where $q\geq N\geq y^{1+\epsilon}$ and any  character $\chi$ mod $q$, with $\chi^{r_f}$  induced from a primitive character of conductor  $Q$, and any integer $b$ we have
\[
 \sum_{n\in I} \chi(f(n))e\bigg( \frac{bn}q\bigg) \ll N / Q^{\eta}.
\] 
\end{corollary}

\begin{corollary} \label{cor: mainnonprincipal2}  For any fixed $\delta,\epsilon>0$   there exists $\eta=\eta(\delta,\epsilon)>0$ and an explicitly determinable  positive integer $\mathcal M=\mathcal M(\delta)$ such that if 
$q\in \mathcal N(y)$  where $y:=q^\delta$ with $(q,\mathcal M)=1$  then for any non-principal character $\chi$  mod $q$
and any  interval $I$ of length $N$ where $q\geq N\geq y^{1+\epsilon}$, we have 
\[
 \sum_{n\in I} \chi(n) \ll N^{1-\eta} .
\]
\end{corollary}

 \begin{corollary}\label{cor: L values}
  Let $q=P_1P_2\cdots$ with $P_1>P_2>\cdots$ the powers of distinct primes dividing $q$.
  Let $\chi$ be a  non-principal character mod $q$. If $P_1$ is a prime then for any $t=q^{o(1)}$ we have 
\[
|L(1+it,\chi)| \leq \max\{\tfrac 12\log P_1(q),\log P_2(q)\}+o(\log q).
\]
If $P_1$ is a prime power then for any $t=q^{o(1)}$ we have 
\[
|L(1+it,\chi)| \leq  \log P_1(q) +o(\log q).
\]
  \end{corollary}

We can also deduce a strong Brun-Titchmarsh inequality, in the same way that  \cite[Theorem 1.5]{XZ} follows from \cite[Lemma 3.6]{XZ},  using Corollary \ref{cor: mainnonprincipal2}.

\begin{corollary}\label{eq:BT-Titchmarsh} 
There exists a sufficiently small $\delta>0$  and  an explicitly determinable  positive integer $\mathcal M=\mathcal M(\delta)$ such that if $q\in \mathcal N(q^\delta)$  with $(q,\mathcal M)=1$  then  
\[
\pi(q^L;q,a) \leq (C_L+o(1)) \frac{x}{\phi(q)\log x} \text{ where } C_L:=\begin{cases}
2 & \text{ for } \frac{12}5 \leq L\leq 8;\\
\frac{5L}{5L-6} & \text{ for } \frac{20}9 \leq L\leq \frac{12}5.
\end{cases}
\]
\end{corollary}

\subsection*{Notation}
For a rational function $h(x)\in \mathbb Q(x)$ we write $h(x)=h_+(x)/h_-(x)$ where $h_+$ and $h_-$ are relatively prime polynomials with integer coefficients. We   define $\deg h:=\max\{ \deg h_+,  \deg h_-\}$, and $\deg_p h$ to be the degree of the reduction of $h \pmod p$.

We observe that if $h_1(x), h_2(x)\in \mathbb Q(x)$ then $\deg( h_1+h_2)\leq \deg h_1+ \deg h_2$,
and $\deg h' , \deg h'/h \leq 2\deg h$. This implies that if $f(x), g(x)\in \mathbb Q(x)$ then
$\deg(g'(x)+Cf'(x)/f(x)) \leq 2 \deg f+2\deg g$.

Suppose that $f(x), g(x)\in \mathbb Q(x)$  and $\chi$ is a primitive character mod $p^m$.
 We define $\chi(f(n))=\chi(f_+(n))\overline{\chi(f_-(n))}$ so that $\chi(f(n))=0$ if $f_-(n)\equiv 0 \pmod p$.
 Similarly  we define $e( \frac{ g(n)} {p^m })=0$ if $g_-(n)\equiv 0 \pmod p$.
 
 Suppose that $f(x)=\frac{c_+}{c_-}\prod_i f_i(x)^{e_i}\in \mathbb Q(x)$  where the $f_i(x)$ are distinct irreducible polynomials in $\mathbb Z[x]$, 
 $c_+$ and $c_- $ are coprime integers, and each $e_i\in \mathbb Z_{\ne 0}$. 
 Let  $\Delta(f)$ equal $r_f$ (as defined in \eqref{eq: Condition}) times the discriminant of   $c_f  \prod_i f_i(x)\in \mathbb Z[x]$
 where $c_f$ is the leading coefficient of $f_+(x)f_-(x)$.
  
  We also define $\Delta(f,g,k)$ to be $\Delta(f)\Delta(g)$ times the product of all of the primes $<2(\deg f+\deg g)+k+16$, for any integer $k\geq 0$.
  
  We define $v_p(r)$ to be the power of $p$ dividing the rational number $r$; that is, if $r=p^ea/b$ where $p$ does not divide the coprime integers $a$ and $b$, and $e\in \mathbb Z$ then $v_p(r)=e$. We also let  $v_p(0)=\infty$.

A \emph{character-exponential sum} is a sum of the form
\[
 \sum_{n  \pmod q}  \chi(f(n)) e\bigg( \frac{ g(n)} {q }  \bigg)    
\]
where $\chi$ is a character mod $q$.
The sum is \emph{complete} if it is a sum   over all $n \pmod q$, and is  \emph{incomplete} if it is a sum   over some subset, usually an interval. We will obtain bounds on incomplete  character-exponential sums, using a complicated differencing technique, in terms of complete  character-exponential sums. These have been widely explored; results in the literature with a prime modulus appear in the form that we need, but with prime power moduli we have decided to give our own proofs although these may be deducible from the literature (see e.g. \cite{Co}) but the deductions are complicated. To prove these estimates we begin with bounds on the number of solutions to polynomial congruences with prime power moduli.

\section{Counting solutions to congruences mod $p^m$}

The main result of this section is the following bound.
    
 \begin{prop} \label{prop: Soln count}
  Suppose that $h(x)=h_+(x)/h_-(x)\in \mathbb Q[x]$   where $h(x)\not\equiv 0 \pmod p$ and $D\ge 1$ is the degree of $h_+(x)$ mod $p$.
  Then
  \[
  \# \{ a \pmod {p^m}: h(a)\equiv 0 \pmod {p^m},\ h_-(a)\not\equiv 0 \pmod p\}  \leq  D p^{m -\lceil \frac m{D}\rceil } .
  \]
  \end{prop}

 Suppose that $h(x)\in \mathbb Z[x]$ and that $h(x)\not\equiv 0 \pmod p$.
  For each $m\geq 1$ let 
  \[
  \mathcal A_m(h):=\{ a \pmod {p^m}: h(a)\equiv 0  \pmod {p^m}\} .
  \]
 
\begin{prop} \label{prop:KeyBound}  For any  $h(x)\in \mathbb Z[x]$ with $d=\deg_p h \ge 1$ and $m\geq 1$ we have 
 \[
  \#  \mathcal A_m(h) \leq d \cdot p^{m-\lceil \frac m{d}\rceil}.
  \]
\end{prop}

Proposition \ref{prop:KeyBound} is best possible:
Let $f(x) = \prod_{i=1}^d (x-ip^r)$ where $p>d \geq 1$. Then $ f(x) \equiv 0 \pmod {p^m}$ where $m=dr+1$
exactly when $x \equiv ip^r \pmod {p^{r+1}}$ for $1\leq i\leq d$.
Therefore the number of solutions here is $d p^{m}/p^{r+1}  = d p^{m-\lceil \frac m{d}\rceil}$.




Cam Stewart \cite[(40), (44)]{St} obtained a similar upper bound but  in terms of $\deg h$ rather than $\deg_p h$, and a slightly weaker family of examples (which yielded $(d-1) p^{m-\lceil \frac m{d}\rceil} $ solutions).
 
 \begin{proof} By induction on $m+d$:
   If $a\in \mathcal A_1$ then $a$ is a root of $h(x) \pmod p$ and so 
  $ \#  \mathcal A_1 \leq \deg_p h=d$.
  Moreover if $a\in \mathcal A_m$  then $a\in \mathcal A_1$ and so 
   \[
  \#  \mathcal A_m \leq \deg_p h \cdot p^{m-1}.
  \]
  This implies the result for $1\leq m\leq d$. So now assume that  $m> d\ge 1$.

 For given $a\in \mathcal A_1$ let $\frac{h^{(j)}(a)}{j!}  =c_j(a) p^{e_j(a)}$  where $p\nmid c_j(a)$ and $e_j(a)\geq 0$.
Let 
\[
r(a):=\min_{j\geq 0} (e_j(a)+j)
\]
 with $J(a):=\{ j: e_j(a)+j=r(a)\}$ so that if $j\in J(a)$ then $j\leq r(a)$.  (Note that $e_0(a)\ge 1$ as $a\in  \mathcal A_1$, and therefore $r(a)\geq 1$.)
 Now if $j<r(a)$ then $1\leq r(a)-j \leq e_j(a)$ and so $p$ divides $h^{(j)}(a)$. Therefore  $(x-a)^{r(a)}$ divides $h(x) \pmod p$ so that $r(a)\leq d<m$.
We have
 \[
 h(a+kp) = \sum_{i\geq 0}    c_{i}(a) p^{e_{i}(a)+i}  k^i  = p^{r(a)} H_a(k)
 \]
 where
 \[
   H_a(k):= \sum_{i\geq 0}    c_{i}(a) p^{e_{i}(a)+i-r(a)}  k^i \equiv   \sum_{j\in J(a)}    c_j(a)    k^j \pmod {p}, 
 \]
so  that $d_a:=\deg_p H_a \leq \max_{j\in J(a)} j\leq r(a)$.

Now     $ h(a+kp)\equiv 0 \pmod {p^m}$ if and only if $H_a(k)  \equiv 0 \pmod {p^{m-r(a)}} $.  
This is determined by the value of $k$ mod $p^{m-r(a)}$ and yet we need to count the number of $k$ mod $p^{m-1}$
and so the number of such $k$ is $p^{r(a)-1}\cdot  \# \mathcal A_{m-r(a)}(H_a)$.
 We deduce that 
 \[
 \# \mathcal A_m(h) = \sum_{a\in \mathcal A_1(h)} p^{r(a)-1} \cdot  \# \mathcal A_{m-r(a)}(H_a).
 \]
 
 Now $\prod_{a\in  \mathcal A_1} (x-a)^{r(a)}$ divides $h(x) \pmod p$ and so 
 \[
  \sum_{a\in \mathcal A_1(h)} d_a\leq   \sum_{a\in \mathcal A_1(h)}  r(a) \leq \deg_p h.
 \]
 We have $d_a+(m-r(a))\leq m < \deg_p h +m$ so we proceed by induction:
  \begin{align*}
 \# \mathcal A_m(h) &= \sum_{a\in \mathcal A_1(h)}  p^{r(a)-1}  \cdot  \# \mathcal A_{m-r(a)}(H_a) \leq
  \sum_{a\in \mathcal A_1(h)}   p^{r(a)-1}  \cdot d_a  \cdot p^{m-r(a)-\lceil \frac {m-r(a)}{d_a}\rceil} 
 \\ &  \leq \sum_{a\in \mathcal A_1(h)}   r(a)  \cdot p^{m-1-\lceil \frac {m-r(a)}{r(a)}\rceil} \leq d\cdot \max_{1\leq r\leq d} p^{m -\lceil \frac {m}{r}\rceil} =d \cdot p^{m-\lceil \frac m{d}\rceil},
\end{align*}  
as $d_a  \leq r(a)$ and $ \sum_{a\in \mathcal A_1(h)}   r(a)\leq d$. 
\end{proof}

 \begin{proof} [Proof of Proposition \ref{prop: Soln count}] By Proposition \ref{prop:KeyBound}, we have
\begin{align*} 
  \# &\{ a \pmod {p^m}: h(a)\equiv 0 \pmod {p^m},\ h_-(a)\not\equiv 0 \pmod p\} \\
       &\leq 
  \# \mathcal A_m(h_+)  \leq D p^{m -\lceil \frac m{D}\rceil } .
\end{align*}
   \end{proof}


\section{Complete character-exponential sum estimates}

Suppose that $f(x), g(x)\in \mathbb Q(x)$  and $\chi$ is a character mod $p^m$ of order $r_\chi\geq 1$.
We need bounds on the \emph{complete character-exponential sums}
\begin{equation}  \label{eq: mixed sum}
 \sum_{n  \pmod{p^m}}  \chi(f(n)) e\bigg( \frac{ g(n)} {p^m }  \bigg)    .
 \end{equation}
 Note that if   $r_\chi$ divides $r_f$  then 
\[
\chi(f(n))=\chi(c) \chi(F(n)^{r_f})=\chi(c) \chi^{r_f}(F(n)) = \chi(c)\chi_0(F(n)),
\]
 where $\chi_0$ is principal, which can only take values $0$ and $ \chi(c)$ which makes the question significantly easier.

\subsection{Complete sums modulo primes $p$}
  Suppose that $\chi$ has order $r= r_\chi$. If 
\[
f(x)\equiv c_1F(x)^r \pmod p \text{ and }
g(x)\equiv G(x)^p-G(x)+c_2 \pmod p
\]
for some $F(x), G(x)\in \mathbb Q(x)$ and constants $c_1,c_2$ with $p \nmid c_1$, then
\[
 \chi(f(n)) e\bigg( \frac{ g(n)}p  \bigg)= \begin{cases}
  \chi(c_1) e( \frac{ c_2}p) &  \text{ if }   f_-(n)f_+(n)g_-(n)\not \equiv 0 \pmod p;\\
    0 &  \text{ if }   f_-(n)f_+(n)g_-(n)\equiv 0 \pmod p;
  \end{cases} 
\]
and so \eqref{eq: mixed sum} has size $\geq p-\deg (f_-f_+g_-)$. 
We call such pairs $(f(x), g(x))$ \emph{degenerate}.

The main result of \cite{CP} (building on work of Weil \cite{We},  \cite[Proposition 12.11]{IK}  and many others)
implies that if $(f(x), g(x))$ is not degenerate and $\chi$ is non-principal  then
\begin{equation} \label{Classic mod p}
 \sum_{n\pmod p} \chi(f(n)) e\bigg( \frac{ g(n)}p  \bigg) \leq (2\deg f + 2\deg g -1) \sqrt{p}.
 \end{equation} 
  \eqref{Classic mod p} also holds in the degenerate case if $g(x)$ is not a constant mod $p$, since then $\deg g \geq p$.
We can therefore deduce the following for primitive characters.

 \begin{prop} \label{prop: KeyEstimate3a}  
If $f(x), g(x)\in \mathbb Q(x)$ and $\chi$ is  a   character mod $p$ of order $r\geq 1$  then
\[
 \bigg|     \sum_{n  \pmod{p}}  \chi(f(n)) e\bigg( \frac{ g(n)} {p}  \bigg)         \bigg|
 \leq \  (2\deg_p f + 2 \deg_p g -1) \,  p^{1/2}  ,
  \]
 unless \text{\rm (i)} $g(x)\equiv  c_2 \pmod p$, and \text{\rm (ii)} $\chi$ is principal or $f(x)\equiv c_1F(x)^r \pmod p$ 
 for some $F(x)\in \mathbb Q(x)$ and constants $c_1,c_2$.
     \end{prop}
     
\begin{proof} This follows from \eqref{Classic mod p} for non-principal characters.
    
 Suppose that  $\chi=\chi_0$ is principal and $g(x)\not\equiv c_2 \pmod p$.
Replace $\chi$ by a primitive character and $f$ by $1$; the difference in value between the two sums is 
$\leq  \#\{ a\pmod p: f(a)\equiv 0 \pmod p\}\leq \deg_p f$, and so we get the upper bound
\[
\deg_p f + ( 2 \deg_p g -1) \,  p^{1/2}\leq \  (2\deg_p f + 2 \deg_p g -1) \,  p^{1/2} 
\]
  by Proposition \ref{prop: KeyEstimate3a} for primitive characters.
  \end{proof}

 We now observe that the set of exceptions in Proposition \ref{prop: KeyEstimate3a} is finite and computable. Recall that $\Delta(f)$ is $r_fc_f$ times the discriminant of the product of the distinct irreducible factors of $f_-(x)f_+(x)$, with $r_f$ as defined by \eqref{eq: Condition}, and $c_f$ the leading coefficient of $f_-(x)f_-(x)$.

 \begin{lemma} \label{lem: f(t)=cF(t)^r} Suppose that $f(x)\in \mathbb Q(x)$ and that
$p$ is a prime for which $f(x)\equiv c_1F(x)^r \pmod p$ for some $F(x)\in \mathbb Q(x)$,  and integers $c_1$ and $r\geq 2$. Then $p$ divides $\Delta(f)$ or $r$ divides $r_f$, with $r_f$ as defined by \eqref{eq: Condition}.
 \end{lemma}
 
\begin{proof}  We factor $f(x)=c\prod_i f_i(x)^{e_i}$ in  $\mathbb Q(x)$, so that $\text{gcd}_i e_i=r_f$.
 Now if $p\nmid \Delta(f)$ then $\prod_i f_i(x)$ will factor  into distinct irreducible factors mod $p$. But then   the exponents of the irreducible factors of 
 $f(x) \pmod p$ will have exponent set $\{ e_i: i\geq 1\}$, sometimes repeated, which has gcd $r_f$, and so if
 $f(x)\equiv c_1F(x)^r \pmod p$ then $r$ divides $r_f$.
 \end{proof}

\subsection{Complete sums modulo prime powers}
The cases with $m\geq 2$  have a distinctly different flavour. Our main goal in this  and the next subsection is to prove the following:

\begin{theorem} \label{thm:chi-expand2} 
Suppose that $f(x),g(x)\in \mathbb Q(x)$ where  $f(x)$ and $g(x)$ are not both constants and $p$ is an odd prime.
For any primitive character  $\chi \pmod {p^m}$ with $m\geq 2$, there exists a non-zero integer $C_{\chi}\in [1,p^{m-1}]$ that is not divisible by $p$ (defined in the proof of Proposition \ref{prop:chi-expand}) such that if we take 
\begin{equation} \label{eq: define H}
h(x)=h_{\chi,f,g}(x):= g'(x)+C_{\chi}f'(x)/f(x)=p^\tau H(x)
\end{equation}
 for some $\tau=\tau_{\chi,f,g}\geq 0$ where $H(x)=H_{\chi,f,g}(x)$ is non-zero mod $p$
 then
\[
\bigg| \frac 1{p^m} \sum_{n \pmod {p^m}}  \chi(f(n)) e\bigg( \frac{ g(n)}{p^{m}} \bigg) \bigg| \leq D p^{ -\big\lceil \frac { m-\tau-1}{2D}\big\rceil },
\]
where $D:= \deg_p H_+$.
 \end{theorem}
 
 If $f(x)$ is not a constant, then $f'/f$ (in reduced form) has a non-constant denominator.  Moreover, over a splitting field, the denominator of $f'/f$ is a product of distinct linear factors. On the other hand if $(x-\alpha)^{-e}\| g(x)$ for some $e\geq 1$ then 
 $(x-\alpha)^{-e-1}\| g'(x)$; that is, every linear factor in the denominator of $g'$ appears to at least the second power and so there is no cancelation with the denominators of $f'/f$. In particular we deduce that $h(x)\ne 0, H(x)\ne 0$ and that $\tau$ is well-defined since $C_\chi\ne 0$. 
 
 If $f$ is a constant then $h(x)=g'(x)$ which is non-zero as $g(x)$ is not a constant, and so again $h(x)\ne 0, H(x)\ne 0$ and  $\tau$ is well-defined.

If $\chi \pmod {p^m}$ is induced from a primitive character $\pmod {p^{\ell}}$ for some $\ell\in [1,m-1]$ then there exists a primitive character  $\psi \pmod {p^m}$ for which $\chi=\psi^{p^{m-\ell}}$. We then define 
$C_{\chi}:\equiv p^{m-\ell} C_{\psi} \pmod {p^{m-1}}$ with $C_\chi \in [1,p^{m-1}]$, where $C_{\psi}$ is defined in the proof of Proposition \ref{prop:chi-expand}.
If $\chi \pmod {p^m}$ is principal, then $C_{\chi}=p^{m-1}$. By the same proof as in the previous paragraph we see that $H(x)\ne 0$ and  $\tau$ is well-defined unless $\chi$ is principal and $g(x)$ is a constant in which case we will let $\tau=m$.
 
 \begin{corollary} \label{cor:chi-expand2} 
Suppose that $f(x),g(x)\in \mathbb Q(x)$ where  $f(x)$ and $g(x)$ are not both constants and $p$ is an odd prime.
For any  character  $\chi \pmod {p^m}$ with $m\geq 2$, define
$h_{\chi,f,g}(x), H_{\chi,f,g}$ and $\tau_{\chi,f,g}$ as in \eqref{eq: define H}  (and as in the previous paragraph for imprimitive $\chi$). Then 
\[
\bigg| \frac 1{p^m} \sum_{n \pmod {p^m}}  \chi(f(n)) e\bigg( \frac{ g(n)}{p^{m}} \bigg) \bigg| \leq D p^{ -\big\lceil \frac { m-\tau-1}{2D}\big\rceil },
\]
where $D:= \deg_p H_+\leq 2(\deg f + \deg g)$.
 \end{corollary}


\subsection{From character values to exponential values}
 Let $\chi=\chi_{p^m}$ be a primitive character $\pmod {p^m}$ for a prime $p\geq 3$.  
Let $v_p(r)=v$ where $p^{-v}r$ does not have $p$ dividing the numerator of denominator of $r$.
Now for $1 \le \ell \le m-1$,  $(1+p^\ell)^r\equiv 1+r p^\ell \pmod {p^{\ell+v_p(r)+1}}$.
We deduce, by taking $r=p^{m-\ell-1}$ and then $r=p^{m-\ell}$
that  $\chi(1+p^\ell)$ is a primitive $p^{m-\ell}$th root of unity, 
and indeed $\chi(1+kp^\ell)$  is also a primitive $p^{m-\ell}$th root of unity for any integer $k$ which is not divisible by $p$. This flexibility in the selection of $k$ allows us to make a choice that will simplify the transition from $\chi$-values to exponential values.

For given integers $m>\ell \geq 1$ we define  $J=J_{p,\ell, m}$ to be the smallest positive integer such that    $p^{j\ell}/j\equiv 0 \pmod {p^m}$ for all $j>J_{p,\ell, m}$. This implies that 
\[
\log (1+p^\ell x)\equiv \sum_{j=1}^{J} \frac{(-1)^{j-1}}j (p^\ell x)^j \pmod {p^m} \text{ and } 
 \sum_{j>J} \frac{ (p^\ell x)^j}{j!}  h^{(j-1)}(a)\equiv 0 \pmod {p^m}
\]
for any given rational function $h(x)\in \mathbb Q(x)$ and any integer $a$ for which $p\nmid h_-(a)$, since $k!$ divides the coefficients of the  Laurent series expansion of $h^{(k)}(x)$ (with $k=j-1$).  Note that 
if $m>\ell\geq \frac{m+1_{p=2}}2$ then $J=1$  and if  $  \frac{m+1_{p=2}}2>\ell\geq \frac{m+1_{p=2}+1_{p=3}}3$ then $J=2$.

\begin{prop} \label{prop:chi-expand} Let $\chi$ be a primitive character $\pmod {p^m}$ where $p^m$ is a prime power $>2$.
There exists an  integer $C_\chi$, coprime with $p$,  such that if $p^\ell>2$ with $1\leq \ell\leq m-1$ then for any $p$-integer $r$ we have
\[
\chi(1+rp^\ell) = e\bigg( \frac{ -C_{\chi} \mathcal L(1+rp^\ell) }{p^{m }}  \bigg) \text{ where } \mathcal L(1+rp^\ell) = \sum_{j=1}^{J_{p,\ell, m}} \frac{(-rp^\ell)^j}{j}.
\]
If $n=a+bp^\ell$ where $\chi(f(a))\ne 0$ then
\[
 \chi(f(n)) e\bigg( \frac{ g(n)}{p^{m}} \bigg) =  \chi(f(a)) e\bigg( \frac{ g(a)}{p^{m}} \bigg)\cdot  
 e\bigg(   \frac{   \sum_{j= 1}^{J_{p,\ell, m}}  \frac{(bp^\ell)^j}{j!}  h^{(j-1)}(a) }  {p^m}  \bigg)
\]
where $h(x):= g'(x)+C_{\chi}f'(x)/f(x)$.
\end{prop}

This proposition is developed from ideas of Postnikov, Gallagher and Iwaniec \cite{Pos, Ga, Iw}.
We also obtain the following:

 \begin{corollary} \label{cor:chi-expand} Suppose that $f(x),g(x)\in \mathbb Q(x)$  and prime $p>2$.
For any primitive character  $\chi \pmod {p^m}$ with $m\geq 2$, let $h(x):= g'(x)+C_{\chi}f'(x)/f(x)$
 and suppose that $h(x)=p^\tau H(x)$ for some $\tau\geq 0$, where $H(x)$ is non-zero mod $p$.
 If   $m>  \ell\geq   \frac{m-\tau }2$ then  
\[
\sum_{n \pmod {p^m}}  \chi(f(n)) e\bigg( \frac{ g(n)}{p^{m}} \bigg) = p^\tau \sum_{a \pmod {p^\ell}}  \chi(f(a)) e\bigg( \frac{ g(a)}{p^{m}} \bigg)\cdot  
\sum_{b \pmod {p^{m-\ell-\tau}}}  e\bigg(   \frac{      b H(a) }  {p^{m-\ell-\tau}}  \bigg).
\]
\end{corollary}

\begin{proof}
If $p^\tau|h(x)$ then $p^\tau|h^{(j)}(x)$ for all $j\geq1$.  Therefore if $2\ell   +\tau \geq m$ then 
$p^m$ divides  $\frac{p^{j\ell}}{j!}  h^{(j-1)}(a)$ for all $j\geq 2$ (since $p>2$). We then substitute in the second formula of Proposition \ref{prop:chi-expand} to obtain
\[
  \chi(f(n)) e\bigg( \frac{ g(n)}{p^{m}} \bigg) =    \chi(f(a)) e\bigg( \frac{ g(a)}{p^{m}} \bigg)\cdot  
  e\bigg(   \frac{      b H(a) }  {p^{m-\ell-\tau}}  \bigg).
\]
The value of the last term depends only on $b \pmod {p^{m-\ell-\tau}}$, and so, summing over $n=a+bp^\ell$, the result follows.
\end{proof}

\begin{proof} [Proof of Theorem \ref{thm:chi-expand2}]  The result is trivial for $ \tau\geq m-1$ so we will assume that
$\tau\leq m-2$. 

The last sum in the displayed equation in Corollary \ref{cor:chi-expand}
  is $0$ unless $H(a)\equiv 0 \pmod  {p^{m-\ell-\tau}}$ in which case it equals $p^{m-\ell-\tau}$, and the value of $H(a)  \pmod  {p^{m-\ell-\tau}}$ depends only on the value of $a \pmod  {p^{m-\ell-\tau}}$. Taking $\ell=   \lceil \frac{m-\tau }2\rceil $ in Corollary \ref{cor:chi-expand}, and using the trivial bound $| \chi(f(a)) e( \frac{ g(a)}{p^{m}} )|\leq 1$ we deduce that 
\begin{align*}
\bigg| \sum_{n \pmod {p^m}}  \chi(f(n)) e\bigg( \frac{ g(n)}{p^{m}} \bigg) \bigg| &\leq
p^{\ell+\tau} \cdot \#\{ a \pmod  {p^{m-\ell-\tau}}: H(a)\equiv 0 \pmod  {p^{m-\ell-\tau}}\} \\
& \leq  D p^{m -\lceil \frac {m-\ell-\tau}{D}\rceil } = D p^{m -\lceil \frac { m-\tau-1}{2D}\rceil },
\end{align*}
by Proposition \ref{prop: Soln count},  where $D=\deg_p H_+$. The equality in the last step is found by a case-by-case analysis.
\end{proof}

  \begin{proof} [Proof of Corollary \ref{cor:chi-expand2}]  If $\chi$ is primitive then the result is Theorem \ref{thm:chi-expand2}.
  
   If $\chi$ is not primitive, then $\chi=\psi^{p^{m-\ell}}$ for some primitive character $\psi \pmod {p^m}$ (with $\ell<m$ minimal). Therefore
 $\chi(f(n)) =\psi^{p^{m-\ell}}(f(n)) =\psi(f(n)^{p^{m-\ell}}) $. We then apply Theorem \ref{thm:chi-expand2} with $\chi$ replaced by $\psi$ and $f(x)$ replaced by $f(x)^{p^{m-\ell}}$. Since $C_\chi:\equiv C_\psi  p^{m-\ell} \pmod {p^{m-1}}$ we have
\[ h_{\psi,f^{p^{m-\ell} },g}(x)
:= g'(x)+C_{\psi}\cdot {p^{m-\ell}} f'(x)/f(x) \equiv g'(x)+C_{\chi}  f'(x)/f(x) \pmod {p^{m-1}},
 \]
 that is, $h_{\psi,f^{p^{m-\ell} },g}(x) \equiv h_{\chi,f,g}(x) \pmod {p^{m-1}}$.  
Since $\tau \le m-2$ we deduce that $H_{\chi,f,g} \equiv H_{\psi,f^{p^{m-\ell}},g}     \pmod p$ and the result now follows from Theorem \ref{thm:chi-expand2}.
 \end{proof}

\begin{proof}[Proof of Proposition \ref{prop:chi-expand}] Let $\chi=\chi_{p^m}$ be a primitive character mod $p^m$.
If $p^\ell> 2$ then the series for $e^{p^\ell}$ converges $p$-adically and so we can truncate it, in fact $1+\sum_{i=1}^{I} p^{i\ell}/i! \pmod {p^m}$ is fixed when  $I\geq I_{p,\ell,m}$ where $I_{p,\ell,m}$ is the smallest integer for which 
$p^{i\ell}/i! \equiv  0\pmod {p^m}$ for all $i>I$. 
 Now as noted above, for $p \nmid k$ and $1 \le \ell \le m-1$, we have $\chi(1+kp^\ell)$ is a primitive $p^{m-\ell}$th root of unity. Thus $\chi(1+\sum_{i=1}^{I} p^{i\ell}/i!)$ is a primitive $p^{m-\ell}$th root of unity so we can define
\[
\chi_{p^m}(e^{p^\ell}):=\chi_{p^m}\bigg(1+\sum_{i=1}^{I} \frac{p^{i\ell}}{i!}  \bigg)=:e\bigg( \frac{ C_{\chi,\ell}}{p^{m-\ell }}  \bigg)
\]
for some integer $C_{\chi,\ell}$ which is not divisible by $p$.
This set up then gives that
\[
e\bigg( \frac{ C_{\chi,1}}{p^{m-\ell }}  \bigg) =e\bigg( \frac{ C_{\chi,1}}{p^{m-1 }}  \bigg)^{p^{\ell-1}} = \chi_{p^m}(e^{p})^{p^{\ell-1}} =
\chi_{p^m}(e^{p^\ell})=e\bigg( \frac{ C_{\chi,\ell}}{p^{m-\ell }}  \bigg)
\]
and so $C_{\chi,\ell} \equiv C_{\chi,1} \pmod {p^{m-\ell}}$ (if $p\geq 3$), and analogously $C_{\chi,\ell}\equiv C_{\chi,2} \pmod {2^{m-\ell}}$ if $p=2$; we denote this common value by $C_\chi$, which  can be taken to be a unique integer between $1$ and $p^{m-1}$.

Now
\[
1+rp^\ell = e^{-\sum_{j\geq 1} \frac{(-rp^\ell)^j}{j} }\equiv e^{-\mathcal L(1+rp^\ell)} \pmod {p^m},
\]
and so
\[
\chi(1+rp^\ell)=\chi(e^{p^\ell})^{-p^{-\ell}\mathcal L(1+rp^\ell)}=e\bigg( \frac{ -C_{\chi}\mathcal L(1+rp^\ell)}{p^{m }}  \bigg).
\]

Now we claim that we have the Taylor expansion
\[
\log \bigg( \frac{f(a+x)}{f(a)} \bigg) = \sum_{j\geq 1}  \frac{x^j}{j!} \bigg( \frac{f'(a)}{f(a)} \bigg)^{(j-1)}.
\]
To see this write $F(x)=\log  ( \frac{f(a+x)}{f(a)}  )= \sum_{j\geq 0}  F^{(j)}(0) \frac{x^j}{j!} $, so
$F^{(0)}(0)=F(0)=\log  ( \frac{f(a)}{f(a)}  )=0$. Then $F'(x)=\frac{f'(a+x)}{f(a+x)}$ so this confirms the $j=1$ coefficient taking $x=0$.
Now $\frac{d}{dx} \frac{f'(a+x)}{f(a+x)} =\frac{d}{d(a+x)}\frac{f'(a+x)}{f(a+x)} $ and so 
\[
F^{(j)}(x)\bigg|_{x=0}  = \bigg( \frac{f'(a+x)}{f(a+x)} \bigg)^{(j-1)} \bigg|_{x=0} = \bigg( \frac{f'(a)}{f(a)} \bigg)^{(j-1)},
\]
which is the claim. Now if $n=a+bp^\ell$ then we let $1+rp^\ell = f(n)/f(a) =  \frac{f(a+x)}{f(a)} \big|_{x=bp^\ell}$, and so
\[
 \chi(f(n))= \chi(f(a)) \cdot \chi ( 1+rp^\ell ) = \chi(f(a)) \cdot  e\bigg(   \frac{ C_{\chi} \sum_{j\geq 1}  \frac{(bp^\ell)^j}{j!} \ ( \frac{f'(a)}{f(a)})^{(j-1)} }  {p^m}  \bigg).
\]
We also have
\[
e\bigg(   \frac{ g(n) }  {p^m}\bigg) = e\bigg(   \frac{ g(a) }  {p^m}\bigg)
e\bigg(   \frac{  \sum_{j\geq 1}  \frac{(bp^\ell)^j}{j!} \ g^{(j)}(a) }  {p^m}  \bigg)
\]
and multiplying these together gives the result.
\end{proof}


\section{Interpreting Corollary \ref{cor:chi-expand2} away from finitely many primes}

\begin{lemma} \label{lem: h=c  mod p} Suppose that $f(x),g(x)\in \mathbb Q(x), C\in \mathbb Z$, and 
 prime $p$ is $>\max\{ \deg f, \deg g\}$.
If  $h(x)=g'(x)+Cf'(x)/f(x)$ is a polynomial mod $p$ then 
$g(x)$ mod $p$ is a polynomial, and either $p|C$ or $f(x)$ mod $p$ is a constant.
\end{lemma}

\begin{proof} If $f(t) \equiv c \prod_i f_i(t)^{e_i}\pmod p$ where the $f_i(t)$ are irreducible then 
$f'(t)/f(t) \equiv \sum_i e_i f_i'(t)/f_i(t) \pmod p$. Now the poles of $g'(t) \pmod p$ cannot be simple, whereas the poles of $f'/f$ are all simple, and so there can be no cancelation of poles between $g'(t)$ and $Cf'(t)/f(t)$. Therefore if $h(t)$ is polynomial then $g'(t) \pmod p$ is a polynomial and either $C\equiv 0 \pmod p$ or each $e_i\equiv 0 \pmod p$ implying that
$f'(t) \equiv 0\pmod p$. 

  Now if  $P(t)$ is an irreducible polynomial for which $P(t)^e\| g_-(t)$ then $P(t)^{e+1}$ is the exact power of $P(t)$ dividing the denominator of $g'(t)$ unless $p$ divides $e$.  But since $g'(t)$  is a polynomial mod $p$, we must have $p|e$. But then either
$p\leq \deg_pg\leq \deg g$, which contradicts the hypothesis or $e=0$. Therefore $g_-(t)$ mod $p$ is a constant; that is $g(t)$ mod $p$ is a polynomial.
   
If $f'(t) \equiv 0\pmod p$ then, by the argument of the last paragraph, we know that $f(t)$ mod $p$ is a polynomial,
and so it must be a $p$th power mod $p$.  Then either  
$p\leq \ \deg f$, a contradiction or $f(t) $ mod $p$ is a constant.
\end{proof}

\begin{lemma} \label{lem: reduction  mod p} If $f(x)\in \mathbb Q(x)\setminus \mathbb Q[x]$ but
$f(x)$ mod $p$ is a polynomial then $p$ divides $\Delta(f)$. Moreover if $f(x)$   is a polynomial of degree $d$
and  $p$ does not divide $\Delta(f)$ then
$f(x)$ mod $p$ is a polynomial of degree $d$,
\end{lemma}

\begin{proof} 
When one reduces $f(x)$ mod $p$ to get a polynomial either the denominator reduces to a constant mod $p$ so that $p$ divides the leading coefficient of the denominator, or the  denominator must divide the numerator and so they must have common factors mod $p$. In both cases we have that  $p$ divides $\Delta(f)$.   Moreover by definition, if $p$ does not divide $\Delta(f)$ then it does not divide the leading coefficient of $f(x)$, and so when we reduce $f$ mod $p$ it has the same degree.
\end{proof}

Combining these two results give the following:

\begin{corollary} \label{cor: reduction  mod p}  Suppose that $f(x),g(x)\in \mathbb Q(x), C\in \mathbb Z$, and 
 prime $p$ is $>\max\{ \deg f, \deg g\}$ where $p\nmid \Delta(f)\Delta(g)$.
 If  $h(x)=g'(x)+Cf'(x)/f(x)$ is a polynomial mod $p$ of degree $d$ then $p>d+1$,
$g(x)$ is a polynomial of degree $d+1$, and either $p|C$ or $f(x)$   is a constant.
In particular, if $h(x)\equiv 0 \pmod p$ then $g(x)$ is a constant and either $p|C$ or $f(x)$   is a constant.
 \end{corollary}
 
 In particular we used here that since $p\nmid \Delta(f)$ we have $\deg f =\deg_p f$ so that $f$ is a constant mod $p$ if and only if it is a constant (and similarly since $p\nmid \Delta(g)$).
 
 \begin{corollary} \label{cor: reduction2  mod p}  Suppose that $f(x),g(x)\in \mathbb Q(x), C\in \mathbb Z$, and 
 prime $p$ is $>\max\{ \deg f, \deg g\}$ where $p\nmid \Delta(f)\Delta(g)$. Let 
 $h(x)=g'(x)+Cf'(x)/f(x)$ and let $\tau=v_p(h(x))$.
 \begin{enumerate}[(a)]
\item  If $f(x)$ and $g(x)$ are constants then $h(x)=0$. Otherwise
\item If $g(x)$ is not constant then $\tau=0$.
\item  If $g(x)$ is a constant and $f$ is not, then   $\tau=v_p(C)$.
\end{enumerate}
  \end{corollary}
  
 The hypothesis of Corollary \ref{cor: reduction2  mod p} implies that if $f'(x)\equiv 0 \pmod p$ then $f(x)$ mod $p$ is a constant since, if not, $f(x)$ mod $p$ is the $p$th power of a non-trivial rational function and so $\deg f\geq \deg_pf\geq p$ contradicting the hypothesis.
  
  With further effort one can show \cite[Corollary 6.2]{CG} that for any $f(x),g(x)\in \mathbb Q(x)$, not both constants and $\tau$ as defined in Corollary    \ref{cor: reduction2  mod p}, $p^\tau$ divides  $\deg_pg(x)$ and $C\deg_pf(x)$, which implies Corollary \ref{cor: reduction2  mod p}.

  We now use Corollary \ref{cor:chi-expand2} to deduce the following
  
 \begin{corollary} \label{cor: reduction3  mod p}   Suppose that $f(x),g(x)\in \mathbb Q(x)$ where  $f(x)$ and $g(x)$ are not both constants, and 
 prime $p$ is $>\max\{ 16, \deg f, \deg g\}$ where $p\nmid \Delta(f)\Delta(g)$. Let $D=  2(\deg f + \deg g)$.
 If $g$ is not constant or $\chi$ is a primitive character mod $p^m$ then 
 \[
\bigg| \frac 1{p^m} \sum_{n \pmod {p^m}}  \chi(f(n)) e\bigg( \frac{ g(n)}{p^{m}} \bigg) \bigg| \leq  D p^{ -\frac {m}{4D}  }.
\]
If $g$ is constant and $\chi$ is induced from a primitive character mod $p^\ell$ with $\ell\geq 1$ then  
\[
\bigg| \frac 1{p^m} \sum_{n \pmod {p^m}}  \chi(f(n)) e\bigg( \frac{ g(n)}{p^{m}} \bigg) \bigg| \leq D p^{ -\frac {\ell}{4D}  }
\]
except if $\ell=1$ and $r_\chi$, the order of $\chi$, divides $r_f$ (as defined in \eqref{eq: Condition}). 
 \end{corollary}
 
 In the cases not covered by this result,
$f$ is not constant,  $g$ is constant and either   $\chi$ is principal or $\ell=1$ and $r_\chi$ divides $r_f$. In these cases
   $  \chi(f(n)) e ( \frac{ g(n)}{p^{m}})=1_{p\nmid f_-(n)f_+(n)g(n)}$ and so the mean value   
 $=p^{-1}\# \{ n \pmod p: f_-(n)f_+(n)g(n)\not\equiv 0 \pmod p\} = 1+O(D/p)$.
 
 \begin{proof}  
 In the first case $\tau=0$ by Corollary \ref{cor: reduction2  mod p}(b), and then 
  \[
\bigg| \frac 1{p^m} \sum_{n \pmod {p^m}}  \chi(f(n)) e\bigg( \frac{ g(n)}{p^{m}} \bigg) \bigg| \leq 
\begin{cases}
D p^{-1/2} & \text{ if } m=1;\\
D p^{ -\big\lceil \frac {m-1}{2D}\big\rceil }& \text{ if } m\geq 2.
\end{cases}  \leq D p^{ -\frac {m}{4D}  } .
\]
by Proposition \ref{prop: KeyEstimate3a}  and Corollary \ref{cor:chi-expand2} respectively.
 
  In the second case $\tau=m-\ell$ by Corollary \ref{cor: reduction2  mod p}(c) but we attack this directly.  Suppose that $\chi$ is induced from   primitive character $\psi$ mod $p^\ell$ with $\ell \ge 1$, so 
 \[
  \frac 1{p^m} \sum_{n \pmod {p^m}}  \chi(f(n)) e\bigg( \frac{ g(n)}{p^{m}} \bigg) =
 e\bigg( \frac{ g(0)}{p^{m}} \bigg) \cdot   \frac 1{p^\ell} \sum_{n \pmod {p^\ell}}  \psi(f(n)) .
 \]
 If $\ell=1$ then we can apply Proposition \ref{prop: KeyEstimate3a} unless
  $f(x)\equiv cF(x)^r \pmod p$.  Since $p\nmid \Delta(f)$ we deduce that this only happens when $f(x)=cF(x)^r$ where $r_\psi=r_\chi$.
  For $\ell\geq 2$ our bound then follows, as above, from Corollary \ref{cor:chi-expand2}. 
   \end{proof}
     

\section{The application in practice} 

Given a prime $p$ and integers $q_1,\dots, q_\ell$ not divisible by $p$ as well as integers $h_{i,j}$, we define $f^*$ and $g^+$   by 
 \[
f^*_{\big[\substack{h_{1,0}, \dots, h_{\ell,0}  \\ h_{1,1} ,\dots, h_{\ell,1}}\big]}(t)
= \prod_{j_1,\dots ,j_\ell\in \{ 0,1\} } f(t+h_{1,j_1}q_1+\cdots +h_{\ell,j_\ell}q_\ell)^{\sigma(j)},
\]
and
\[
g^+_{\big[\substack{h_{1,0}, \dots, h_{\ell,0}  \\ h_{1,1} ,\dots, h_{\ell,1}}\big]}(t)
= \sum_{j_1,\dots ,j_\ell\in \{ 0,1\} } \sigma(j) \cdot g(t+h_{1,j_1}q_1+\cdots +h_{\ell,j_\ell}q_\ell) ,
\]
with $\sigma(j)=(-1)^{j_1+\cdots +j_\ell}$.   We wish to apply Theorem \ref{thm:chi-expand2}  and Corollary \ref{cor:chi-expand2}   with $f=f^*$ and $g=g^+$ (and with $g=g^++b$ for an arbitrary constant $b$).
By definition, the leading coefficient of $ f^*_{\big[\substack{h_{1,0} \dots, h_{k,0}  \\ h_{1,1}, \dots, h_{k,1}}\big]}(t)$ will be 1 for all $k\geq 1$.

\subsection{The $p$-divisibility of the rational function differences}

This is the key new result in this article.

\begin{prop} \label{prop: vp} Let $H(t)\in \mathbb Q(t)$ with $H(t)\not\equiv 0 \pmod p$. Let $p$ be a prime $\geq  k+ \deg H$ which does not divide $q_1\cdots q_k$,  such that   if $H(t)$ reduces to a polynomial mod $p$ then $\deg_p H\geq k$. Then 
\[
v_p\bigg( H^+_{\big[\substack{h_{1,0} \dots, h_{k,0}  \\ h_{1,1}, \dots, h_{k,1}}\big]}(t)\bigg)  = 
\sum_{i=1}^k v_p( h_{i,1}-h_{i,0})=:v_p(\mathbf{h}).
\]
Moreover for any integer $b$ we have 
\[
v_p\bigg( H^+_{\big[\substack{h_{1,0} \dots, h_{k,0}  \\ h_{1,1}, \dots, h_{k,1}}\big]}(t) +b\bigg)  =  \min \{ v_p(\mathbf{h}), v_p(b) \}  
\]
assuming that   if $H(t)$ reduces to a polynomial mod $p$ then $\deg_p H> k$.
\end{prop}

\begin{proof} Write $(h_{i,1}-h_{i,0})q_i=u_ip^{e_i}$ where $p\nmid u_i$, for $1\leq i\leq k$.

We begin with the simplest case, that $H(t)$ reduces to a polynomial mod $p$ of degree $D=\deg_p H \geq k$. If 
$x(t):=H^+_{\big[\substack{h_{1,0} \dots, h_{\ell,0}  \\ h_{1,1}, \dots, h_{\ell,1}}\big]}(t) $ then
\[
 H^+_{\big[\substack{h_{1,0} \dots, h_{\ell+1,0}  \\ h_{1,1}, \dots, h_{\ell+1,1}}\big]}(t) = x^+_{\big[\substack{  h_{\ell+1,0}  \\   h_{\ell+1,1}}\big]}(t), 
\]
and if $x(t)=x_dt^d+\ldots$ then $x^+_{\big[\substack{  h_{j,0}  \\   h_{j,1}}\big]}(t)  = (h_{j,1}-h_{j,0})q_j (dx_d t^{d-1}+\ldots)$. Therefore, by induction on $j$ we have that if $H(t)\equiv c\, t^D+\dots \pmod p$ where $p\nmid c$ then, since $D\geq k$,
\[
 H^+_{\big[\substack{h_{1,0} \dots, h_{k,0}  \\ h_{1,1}, \dots, h_{k,1}}\big]}(t) \equiv  \prod_{j=1}^k (h_{j,1}-h_{j,0})  q_j 
  \cdot \bigg(\frac{D!}{(D-k)!} c\, t^{D-k}+\dots   \bigg) \pmod {p^{1+v_p(\mathbf{h})}}.
\] 
Since $p\nmid c\, q_1\cdots q_k$ and $p>D$ so that $p\nmid \frac{D!}{(D-k)!} $ we deduce that 
\[
v_p\bigg( H^+_{\big[\substack{h_{1,0} \dots, h_{k,0}  \\ h_{1,1}, \dots, h_{k,1}}\big]}(t)\bigg) =v_p(\mathbf{h}),
\]
 as desired, and   that $v_p\bigg( H^+_{\big[\substack{h_{1,0} \dots, h_{k,0}  \\ h_{1,1}, \dots, h_{k,1}}\big]}(t) +b\bigg)  =  \min \{ v_p(\mathbf{h}), v_p(b) \}$ if $D>k$
 and so the result follows for $H(t)$ that reduce to polynomials mod $p$.

Now we can assume that $H(t)$ is a rational function that does not reduce to a polynomial mod $p$.
We  organize the $h_{i,j}$ so that  
$e_i=0$ for  $i\leq \ell$ and $e_i\geq 1$ for $i> \ell$. We write
$s(t):=H^+_{\big[\substack{h_{1,0} \dots, h_{\ell,0}  \\ h_{1,1}, \dots, h_{\ell,1}}\big]}(t)\pmod p$ for convenience.

We can write 
\[
H(t) \equiv H_0(t)  + \sum_P \frac{Q(t)}{P(t)^e}  \pmod p
\]
where $H_0(t)$ is a polynomial, where   $P(t), Q(t)\in \mathbb F_p[t]$ with $\deg Q< e\deg P$ in each summand, and the $P(t)$ are distinct  irreducible polynomials. There is at least one such fractional summand, $Q(t)/P(t)^e$.

If the map $t\to t+a$ for some $a\not\equiv 0 \pmod p$ rotates the roots of $P(t)$ then so does $t\to t+ka$ for all $k$ and so $P(t)=R(t^p-t)$ for some polynomial $R(t)$, but then $p\leq \deg P\leq \deg H$ contradicting the hypothesis. 
Therefore the sum above contains a non-zero subsum of the form
$r(t)=\sum_{a=0}^{p-1} Q_a(t+a)/P(t+a)^e$.  (Note that $Q_a(t)$ depends on $a$, but the coefficients are not necessarily  polynomial in $a$.) Then $r^+(t)$, the contribution of the rational function $r(t)$
to the expansion of    $s(t)$ is 
\[
\sum_{ j_1,\dots ,j_\ell\in \{ 0,1\}  } (-1)^{j_1+\cdots +j_\ell}  r(t+h_{1,j_1}q_1+\cdots +h_{\ell,j_\ell}q_\ell)
\equiv 
\sum_{m=0}^{p-1} r(T+m)  \sum_{\substack {I \subset \{1,\dots ,\ell \} \\   \sum_{i\in I} u_i \equiv m \pmod p}} (-1)^{|I|} 
\pmod p
\]
where $T=t+h_{1,0}q_1+\cdots +h_{\ell,0}q_\ell$. Therefore
\begin{align}
r^+(t)&\equiv \sum_{m=0}^{p-1}  \sum_{a=0}^{p-1} \frac{Q_a(T+a+m)}{P(T+a+m)^e}  \sum_{\substack {I \subset \{1,\dots ,\ell \} \\   \sum_{i\in I} u_i \equiv m \pmod p}} (-1)^{|I|} 
\pmod p \notag \\
&\equiv \sum_{n=0}^{p-1} \frac{1}{P(T+n)^e}
\sum_{a=0}^{p-1} \sum_{\substack {I \subset \{1,\dots ,\ell \} \\   a+\sum_{i\in I} u_i \equiv n \pmod p}} (-1)^{|I|}  Q_a(T+n)  \pmod p.
 \label{eq: Coeffs0?}
\end{align}

If this is $\equiv 0 \pmod p$ then each of the coefficients is $\equiv 0 \pmod p$, that is, taking $t=T+n$,
\[
\sum_{a=0}^{p-1} \sum_{\substack {I \subset \{1,\dots ,\ell \} \\   a+\sum_{i\in I} u_i \equiv n \pmod p}} (-1)^{|I|}  Q_a(t) \equiv 0 \pmod p
\]
for each $n \pmod p$. This can be re-expressed as 
\begin{align*} 0&\equiv  \sum_{a \pmod p} Q_a(t) x^{a} \cdot \sum_{I \subset \{1,\dots ,\ell \}  } (-1)^{|I|} x^{\sum_{i\in I} u_i}\\
& \equiv  
  \sum_{a \pmod p} Q_a(t) x^{a} \cdot \prod_{i=1}^\ell (1-x^{u_i}) \pmod { (p,x^p-1)}.
\end{align*}
  By assumption the $u_i\not\equiv 0 \pmod p$  so that $(1-x^{u_i},x^p-1)=x-1$.
Now $\frac{1-x^{u_i}}{1-x}$ is a unit $ \pmod { (p,x^p-1)}$ and $(x-1)^p\equiv x^p-1 \pmod p$,
so the above is equivalent to
\[
(1-x)^\ell \sum_{a=0}^{p-1}  Q_a(t) x^{a}  \equiv 0 \pmod { ( p, (x-1)^p )}.
\]
Therefore we have the linear recurrence relation 
\[
\sum_{j=0}^\ell Q_{a+j \mod p}(t) \binom \ell j (-1)^{\ell-j}  \equiv 0 \pmod p \text{ for all } a\geq 0.
\]
Since the characteristic polynomial for the linear recurrence relation is $(x-1)^\ell$, we can deduce that
$Q_a(t)\equiv R(a) \pmod p$ where $R(x)\in \mathbb F_p[t][x]$ is a polynomial of degree $\ell-1$, and so either this polynomial equals $0$ and so all the $Q_a(t)\equiv 0 \pmod p$ (which is impossible by the definition of $r(t)$), or there are at most $\ell-1$ zeros mod $p$, and so $Q_a(t)\not\equiv 0 \pmod p$ for $\geq p+1-\ell$ values of $a$. But the denominator of $r(t)$ has degree
at least the number of non-zero $Q_a(t)$ and so 
$p+1-\ell\leq \deg r_-\leq \deg H_-\leq \deg H$ which implies that 
$p\leq \ell+\deg H-1\leq k+\deg H-1$,   contradicting the hypothesis.

We deduce that $r^+(t)$, as represented in \eqref{eq: Coeffs0?}, has some non-zero terms, which are also 
non-zero terms of $s(t)$. That is,  $s(t)\pmod p$  has a summand of the form 
\[
q(t)/P(t)^f \text{ for some } 1\leq f\leq e\text{ with } \deg q< f \deg P \text{ and } (q(t),P(t))=1,
\]
where $P(t)$ is an irreducible polynomial.
If $\ell=k$ this implies that $p\nmid s(t)+b$ for any $b \pmod p$ which gives the result in this case.

Let $m=k-\ell\geq 1$ so that   $s^{(m)}(t)\pmod p$ has a summand of the form $r(t)/P(t)^{f+m}$ unless $p$ divides $f(f+1)\cdots (f+m-1)$, in which case $p\leq f+m-1\leq e+k-1\leq \deg h+k-1$ which contradicts the hypothesis, and so $p\nmid s^{(m)}(t)$.
 
 Now define 
 \[
 s^+:= s^+_{\big[\substack{h_{\ell+1,0} \dots, h_{k,0}  \\ h_{\ell+1,1}, \dots, h_{k,1}}\big]}(t) = H^+_{\big[\substack{h_{1,0} \dots, h_{k,0}  \\ h_{1,1}, \dots, h_{k,1}}\big]}(t).
 \]
Let $S(t)=s(t+h_{\ell+1,0}q_{\ell+1}+\cdots + h_{k,0}q_k)$  so that 
 \begin{align*}
s^+(t) &= S^+_{\big[\substack{0, \dots, 0 \\ u_{\ell+1}p^{e_{\ell+1}}, \dots, u_kp^{e_k}}\big]}(t)= \sum_{I\subset \{ \ell+1,\dots,k\}} (-1)^{|I|} \cdot S\bigg(t+\sum_{i\in I} u_ip^{e_i}\bigg) \notag \\
& =\sum_{I\subset \{ \ell+1,\dots,k\}} (-1)^{|I|}  \sum_{n\geq 0}  \frac{S^{(n)}(t)}{n!} \bigg(\sum_{i\in I} u_ip^{e_i} \bigg)^n   \\
&=  (-1)^m\prod_{i=\ell+1}^k u_i p^{e_i} \bigg( S^{(m)}(t)+
 \sum_{n\geq m+1}  S^{(n)}(t) \sum_{\substack{ j_{\ell+1}+\cdots+j_k=n \\  j_1,\dots,j_k\geq 1}}  \prod_{i=\ell+1}^k \frac{(u_ip^{e_i})^{j_i-1}}{j_i!}  \bigg)  ,
\end{align*}
as 
\[
\sum_{I\subset \{ \ell+1,\dots,k\}} (-1)^{|I|}  \bigg(\sum_{i\in I} u_ip^{e_i}  \bigg)^n =\begin{cases} 0 & \text{ if }n<m;\\
 (-1)^m m! \prod_{i=\ell+1}^k u_i p^{e_i}& \text{ if }n=m.
\end{cases}
\]

Now $v_p(S^{(m)}(t))=0$ since $p\nmid s^{(m)}(t)$.  When $n>m$  there is some $j_i>1$ in each of the summands for the $n$th term
and
 $p$ divides $(p^e)^{j-1}/j!$ whenever $j\geq 2$ as $p>2$, so the $n$th term is $0 \pmod p$. Therefore
\[
v_p(s^+(t)) = \sum_i e_i, 
\]
which implies that $v_p(s^+(t)+b) =\min \{ v_p(s^+(t)), v_p(b)\}$ and so the result follows.
 \end{proof}

 


\subsection{When $f^*$ is an $r$th power mod $p$}

\begin{lemma} \label{lem: rthpowers+} Let $f(t)\in \mathbb Q(t)$.
Assume that $p$ is a prime that does not divide $q_1\cdots q_k$. 
Let $r$ be an integer $>1$ that is not divisible by $p$.
If $f^*_{\big[\substack{h_{1,0} \dots, h_{k,0}  \\ h_{1,1}, \dots, h_{k,1}}\big]}(t)$ is a constant times  an $r$th power mod $p$ 
then either  $p$ divides $\prod_j (h_{j,0}- h_{j,1})$, or 
  $f(t)\equiv F(t)^r H(t^p-t) \pmod p$ for some $H(t)\in \mathbb Q(t)$.
 If $H(t) \pmod p$ is not a constant then $p\leq 2\deg f$.
\end{lemma}

\begin{proof}   We reduce $f(t)$ mod $p$. Let $g(t)$ be an irreducible polynomial  mod $p$  that divides the numerator or denominator of  $f(t) \pmod p$ and let $ \prod_{m \pmod p} g(t+m)^{e_m}\| f(t)$ be the exact power of all the translates of $g(t)$ that divide the numerator or denominator of $f$. Then 
$      \prod_{n}  g(t+n)^{h_n} \| \prod_{i=1}^d \frac{ f(t+a_i)}{  f(t+b_i)}$
where $h_n = \sum_i e_{n-a_i} -  \sum_j e_{n-b_j}$ (here all subscripts are taken mod $p$). The $g$ factors yield an $r$th power mod $p$ if and only if each $h_n\equiv 0 \pmod r$; in other words,
\[
\sum_{m \pmod p} e_m x^{m} \cdot \bigg( \sum_i x^{a_i} -  \sum_j x^{b_j}\bigg) \equiv 0 \pmod {(r,x^p-1)}.
\]
For $f=f_{\big[\substack{h_{1,0},\dots, h_{k,0}  \\ h_{1,1},\dots, h_{k,1}}\big]}(t)$ this gives 
\[
\sum_{m \pmod p} e_m x^{m} \cdot \prod_{j=1}^k (x^{h_{j,0}q_j}-x^{h_{j,1}q_j})\equiv 0 \pmod {(r,x^p-1)}.
\]
We have a solution if $h_{\ell,0}\equiv h_{\ell,1} \pmod p$ for some $\ell, 1\leq \ell\leq k$. If not
 substitute in $x=\zeta_p$, a primitive $p$th root of unity, so that 
 \[
\zeta_p^c\cdot \sum_{m \pmod p} e_m \zeta_p^{m} \cdot \prod_{j=1}^k (\zeta_p^{b_j}-1)\equiv 0 \pmod r
\]
where $b_j=(h_{j,0} - h_{j,1})q_j\not\equiv 0 \pmod p$ and $c=\sum_j h_{j,1}q_j$. This is a congruence in the ring
$\mathbb Z[\zeta_p]$. Now each $\zeta_p^{b}-1$ is a $\zeta_p-1$ times a unit, $\zeta_p^c$ is a unit, and $\zeta_p-1$
divides $p$, and therefore all these contributions are coprime to $r$. Hence $r$ divides
$\sum_{m \pmod p} e_m \zeta_p^{m}$ which implies that the  $e_m$ are all congruent, say to $e$ mod $r$. This implies that
 \[
  \prod_{m} g(t+m)^{e_m} \equiv  G(t)^r \bigg(\prod_{m=0}^{p-1} g(t+m)\bigg)^e \pmod p
  \]
 for some   $G(t)$, a product of $g$-translates, while $\prod_{m=0}^{p-1} g(t+m)$ is a function of $t^p-t$.\footnote{Note that 
 $\prod_{m=0}^{p-1} (t-\alpha+m)=   (t-\alpha)^p-(t-\alpha)=t^p-t-(\alpha^p-\alpha)$.}
 But this must be true for any irreducible factor of $f(t)$ and so $f(t)$ is an $r$th power times  a function of $t^p-t$.
 
 If $e\equiv 0 \pmod r$ then $ \prod_{m} g(t+m)^{e_m}$ is an $r$th power; if this holds for all factors of $f$ then $f(t)$ is a constant times an $r$th power mod $p$.  Otherwise there is some such $g$ with $e\not\equiv 0 \pmod r$ with say $1\leq e\leq r-1$. If $N=\#\{ m: e_m>0\}$ then
 $\deg f\geq \deg  \prod_{m} g(t+m)^{e_m} \geq \max\{  Ne , (r-e)(p-N)\}  \geq p/2$.
 \end{proof}
 
 We deduce from Lemmas \ref{lem: rthpowers+}  and \ref{lem: f(t)=cF(t)^r} 
   
\begin{corollary} \label{cor: rthpowers+} Suppose that $f(x)\in \mathbb Q(x)$ and \eqref{eq: Condition} holds.
If $p$ is a prime that does not divide $\Delta(f,g,k)\cdot q_1\cdots q_k$ then  $f^*_{\big[\substack{h_{1,0} \dots, h_{k,0}  \\ h_{1,1}, \dots, h_{k,1}}\big]}(t)$
 is a constant times  an $r$th power mod $p$ for some integer $r\geq 2$ which is not divisible by $p$
if and only if   $p$ divides $\prod_j (h_{j,0}- h_{j,1})$ (in which case $f^*\equiv1 \pmod p$), or 
 $r$ divides $r_f$. 
 \end{corollary}

\subsection{The main complete exponential-character sum for difference functions, in context}

Recall that  $h(x)=h_{\chi,f,g}(x)= g'(x)+Cf'(x)/f(x) =p^\tau H(x)$ for some integer $C=C_\chi$ where $p\nmid H(x)$.
For any choice of the $h_{i,j}\in \mathbb Z$ write
 \[
 f^*=f^*_{\big[\substack{h_{1,0}, \dots, h_{k,0}  \\ h_{1,1} ,\dots, h_{k,1}}\big]} \text{ and } g^+=g^+_{\big[\substack{h_{1,0}, \dots, h_{k,0}  \\ h_{1,1} ,\dots, h_{k,1}}\big]}.
 \]
We wish to apply our results to
\[
\frac 1{p^m} \sum_{n \pmod {p^m}}  \chi ( f^*(n) ) e\bigg( \frac{ g^+(n)+bn}{p^{m}} \bigg).
\]
Noting that 
\[
\bigg(\frac{f'}{f} \bigg) ^+_{\big[\substack{h_{1,0} \dots, h_{k,0}  \\ h_{1,1}, \dots, h_{k,1}}\big]}(x) = (f^*)'_{\big[\substack{h_{1,0} \dots, h_{k,0}  \\ h_{1,1}, \dots, h_{k,1}}\big]}(x)/f^*_{\big[\substack{h_{1,0} \dots, h_{k,0}  \\ h_{1,1}, \dots, h_{k,1}}\big]}(x),
\]
we deduce that 
\[
 h^+_{\big[\substack{h_{1,0} \dots, h_{k,0}  \\ h_{1,1}, \dots, h_{k,1}}\big]}(x)=(g')^+_{\big[\substack{h_{1,0} \dots, h_{k,0}  \\ h_{1,1}, \dots, h_{k,1}}\big]}(x)+C\bigg(\frac{f'}{f} \bigg) ^+_{\big[\substack{h_{1,0} \dots, h_{k,0}  \\ h_{1,1}, \dots, h_{k,1}}\big]}(x) =p^\tau H^+_{\big[\substack{h_{1,0} \dots, h_{k,0}  \\ h_{1,1}, \dots, h_{k,1}}\big]}(t) .
 \]
Under the hypotheses of  Proposition \ref{prop: vp} we then obtain that, for any integer $b$, we have 
\[
v_p\bigg( H^+_{\big[\substack{h_{1,0} \dots, h_{k,0}  \\ h_{1,1}, \dots, h_{k,1}}\big]}(t)  \bigg)  =  \ v_p(\mathbf{h}) 
\]
and so
\[
v_p\bigg( h^+_{\big[\substack{h_{1,0} \dots, h_{k,0}  \\ h_{1,1}, \dots, h_{k,1}}\big]}(t) +b\bigg)  =  \min \{ \tau+v_p(\mathbf{h}), v_p(b) \}  =:\tau'.
\]
We will use the notations $\tau$ and $\tau'$ in the proof the following result.

 \begin{prop} \label{prop: vp rev} Suppose that $f(x),g(x)\in  \mathbb Q(x)$ and $f(x)$ and $g(x)$ are not both constants, and let  $D=2(\deg f+\deg g)$.
 Assume $p \nmid q_1\cdots q_k\Delta(f,g,k)$. 
 If $\chi$ is a character mod $p$ then
  \[
\bigg| \frac 1{p} \sum_{n \pmod {p}}  \chi ( f^*(n) ) e\bigg( \frac{ g^+(n)+bn}{p} \bigg) \bigg| 
\leq 2^{k+1}Dp^{-\frac{1-v_p(\mathbf{h})}2},
 \]
and if $\chi$ is a character mod $p^m$ for any $m\geq 1$, then
  \[
\bigg| \frac 1{p^m} \sum_{n \pmod {p^m}}  \chi ( f^*(n) ) e\bigg( \frac{ g^+(n)+bn}{p^m} \bigg) \bigg| 
\leq 2^{k}Dp^{-\frac{m-v_p(\mathbf{h})}{2^{k+2}D}},
 \]
unless $g$ is a polynomial of degree $\leq k$  and   $\chi^{r_f}$ is induced from a character of conductor $p^\ell$ with $\ell<m$. In these cases we get the bound $\leq 2^{k}Dp^{-\frac{\ell-v_p(\mathbf{h})}{2^{k+2}D}}$.
 \end{prop}
  
 We take the principal character to be induced from the character of conductor $p^0$.

  \begin{proof} When $m=1$ we need consider only $v_p(\mathbf{h})=0$ else the result is trivial.
 One gets a bound
 $\leq 2^{k+1}Dp^{-\frac{1}2}\leq 2^{k}Dp^{-\frac{1-v_p(\mathbf{h})}{2^{k+2}D}}$ by Proposition \ref{prop: KeyEstimate3a}   
 unless $g^+(x)+bx\equiv c_2 \pmod p$ (which implies that $g(x)$ mod $p$ is a polynomial of degree $\leq k+1$)\footnote{To see that $g(x)$ mod $p$ must be a polynomial proceed as in the middle part of the proof of 
 Proposition \ref{prop: vp} (with $H(x)$ there replaced by $g(x)$ here, and with $\ell=k$ --- see there how consideration of the form of $r(t)$ leads to this conclusion for $s(t)$) to deduce that $g^+(x)$ mod $p$ is not a polynomial, a contradiction.}  and, either $\chi$ is principal or   there exists $F(x)\in \mathbb Q(x)$ with
 $f^*(x)\equiv c_1 F(x)^{r_\chi} \pmod p$ where $r_\chi$ is the order of $\chi$ which implies that $r_\chi$ divides $r_f$ by Corollary \ref{cor: rthpowers+}, whence $\chi^{r_f}$ is principal. Therefore $\ell=0$ and the result is trivial.

Now let $m\geq 2$  and $\chi$ be a character mod $p^m$. We separate into three cases:

(I):\ If $g(x)$ is not constant, or if $g(x)$ is a constant and $\chi$ is primitive then  $\tau=0$ by Corollary \ref{cor: reduction2  mod p}. If $H(t)\pmod p$ is a polynomial then $g(t)$ is a polynomial, and $f$ is a constant or $\chi$ is imprimitive, by Corollary \ref{cor: reduction  mod p}.
 Then $\deg_pH\geq k$ if and only if $\deg_p g\geq k+1$, if and only if $\deg g\geq k+1$ (since $p\nmid \Delta(g)$). Therefore
 Proposition \ref{prop: vp} is   applicable unless $g$ is a polynomial of degree $\leq k$, and $f$ is a constant or $\chi$ is imprimitive. If
 Proposition \ref{prop: vp} is   applicable then we can apply Corollary \ref{cor:chi-expand2} with $f$ and $g$ replaced by $f^*$ and $g^++bx$, and then $\tau$ replaced by $\tau'=  \min \{  v_p(\mathbf{h}), v_p(b) \} \leq v_p(\mathbf{h})$ to obtain  the bound
$\leq 2^{k}Dp^{-\frac{m-1-\tau'}{2^{k+1}D}}  \leq 2^{k}Dp^{-\frac{m-v_p(\mathbf{h})}{2^{k+2}D}}$.

(II):\  If $g(x)$ is a polynomial of degree $\leq k$, $f(x)$ is not a constant and $\chi$ is induced from a character of conductor $p^\ell$ with $\ell<m$ then  $\tau=m-\ell$ above. Now 
 $f(t) \pmod p$ is not a constant as $p\nmid \Delta(f)$. Therefore Proposition \ref{prop: vp} is   applicable and so
 we can apply Corollary \ref{cor:chi-expand2} with $f$ and $g$ replaced by $f^*$ and $g^++bx$, and then $\tau$ replaced by 
 $\tau'=  \min \{ m-\ell+ v_p(\mathbf{h}), v_p(b) \} \leq m-\ell+v_p(\mathbf{h})$ to obtain  the bound
$\leq 2^{k}Dp^{-\frac{\ell-1-v_p(\mathbf{h})}{2^{k+1}D}}  \leq 2^{k}Dp^{-\frac{\ell-v_p(\mathbf{h})}{2^{k+2}D}}$.

(III):\   If $g$ is a polynomial of degree $\leq k$ and $f$ is a constant   then $\chi^{r_f}$ is principal by definition, so that $\ell=0$ and the result is trivial. 
\end{proof}


\section{Basic inequalities for 1-bounded periodic functions} 

The heart of our proof lies in precise van der Corput differencing which is explained in this section. We will eventually apply the results of this section by taking (for $a=a_q$)
\[
a_q(n)=\chi_q(f(n)) e\bigg( \frac{ g(n)}q  \bigg).
\]

Let $a_q(\cdot)$ be a function of period  $q$ with each $|a_q(n)|\leq 1$, 
and if $q=Qr$ where $(Q,r)=1$ then $a_q(n)=a_Q(n) a_r(n)$.
Our goal in this section is to develop methods to bound incomplete sums of $a_q(\cdot)$ in terms of complete sums, perhaps ``twisted''; that is, for  
  intervals $I$, of length $N$, techniques to bound 
  \[
\mathbb E_q(a):= \frac 1N \sum_{n\in I} a_q(n) .
\]

Using the Chinese Remainder Theorem and Fourier analysis we will prove
 
 \begin{lemma} \label{Lem: BdsInInterval}
If $N$ is the length of the interval $I$ then 
\begin{equation}\label{eq: CRT}
\frac 1N \bigg|  \sum_{n\in I} a_q(n) \bigg|   \ll \bigg( 1+ \frac {q \log q}N\bigg)  \cdot \  \prod_{ p^e\| q} \ \max_{b \pmod {p^e} }  \bigg|  \frac{1}{p^e} \sum_{r \pmod {p^e}} a_{p^e}(r)     
e\bigg( \frac{br}{p^e}\bigg) \bigg| .
\end{equation}
\end{lemma}

Since we cannot hope for better than square-root cancellation, this bound (which works for general moduli $q$)  can only be non-trivial for $N>q^{1/2+o(1)}$. The following lemma will allow us to work with shorter intervals provided $q$ factors into small pieces.
Unfortunately we need to introduce some substantial notation to understand how we apply van der corput-like differencing:

Given integers $q_1,\dots, q_\ell\geq 1$ and $h_{i,j}\geq 1$ with $1\leq i\leq \ell, 0\leq j\leq 1$ we define
\begin{align*}
a_{\big[\substack{h_{1,0},h_{2,0},\dots, h_{\ell,0}  \\ h_{1,1},h_{2,1},\dots, h_{\ell,1}}\big]}(n)
&= \prod_{\substack{j_1,\dots ,j_\ell\in \{ 0,1\} \\   j_1+\cdots +j_\ell \equiv 0 \pmod 2}} a(n+h_{1,j_1}q_1+\cdots +h_{\ell,j_\ell}q_\ell) 
\\
& \times \prod_{\substack{j_1,\dots ,j_\ell\in \{ 0,1\} \\   j_1+\cdots +j_\ell \equiv 1 \pmod 2}} \overline{a(n+h_{1,j_1}q_1+\cdots +h_{\ell,j_\ell}q_\ell) }
\end{align*}
a product of $a$ and $\overline{a}$ each evaluated at
 $2^{\ell-1}$ linear factors.
Given integers $M_1,\cdots, M_\ell\geq 1$ we now define
\[
\textbf{E}_{Q}^{(i)} (a_{(M_1,\cdots, M_\ell) } ) :=
\frac 1{M_1^2\cdots M_\ell^2} \sum_{\substack{1\leq h_{j,0} \ne h_{j,1} \leq M_j \\ \text{for } 1\leq j\leq \ell}}
\bigg| \mathbb{E}_{Q}\bigg(  a_{\big[\substack{h_{1,0},\dots, h_{\ell,0}  \\ h_{1,1},\dots, h_{\ell,1}}\big]}\bigg) \bigg|^i
\]
for $i=1$ and $2$ and let $\textbf{E}_{Q} =\textbf{E}_{Q}^{(1)}$, so that by Cauchying
\[
\textbf{E}_{Q}  (a_{(M_1,\cdots, M_\ell) } )^2\leq \textbf{E}_{Q}^{(2)} (a_{(M_1,\cdots, M_\ell) } ).
\]

\begin{corollary} \label{cor: First thm}
Given $N$, suppose that integer  $q=q_1q_2\cdots q_kQ$ where the $q_i$ and $Q$ are pairwise coprime, and
each $q_j\leq N^{1-\epsilon}$.  Then
  \[
|\mathbb E_q(a)| \ll  k  N^{-\epsilon/2^{k}}  +    \textbf{E}_{Q}  (a_{(M_1,\cdots, M_k) } )^{1/2^k}.
\]
\end{corollary}

This is an immediate Corollary of Proposition \ref{prop: First thm} below.  To obtain a good bound on 
$|\mathbb E_q(a)|$ we will need a good bound on the 
\[
\bigg| \mathbb E_{Q}\bigg(  a_{\big[\substack{h_{1,0},\dots, h_{\ell,0}  \\ h_{1,1},\dots, h_{\ell,1}}\big]}\bigg) \bigg|
\]
on average, and to do this we use the decomposition in Lemma \ref{Lem: BdsInInterval} as we will discuss in section \ref{sec: Together}.

\subsection{Summing up displacements}

\begin{lemma}  \label{lem: BasicIneq}
Suppose that $q=rQ$ with $(r,Q)=1$, and let $M=M_r=\lfloor (N/r)^{2/3}\rfloor \geq 5$. Then
\begin{equation} \label{BasicInequality}
|\mathbb E_q(a)|^2\leq \frac 5{M_r}
+  \frac 2{M_r^2} \sum_{1\leq i\ne j\leq M_r}  |\mathbb E_Q(A_{(ir,jr)})|  ,
\end{equation}
where $A_ {(ir,jr)}(n):= a(n+ir) \overline{a(n+jr)} $.
\end{lemma}

\begin{proof}
For any positive integer $\Delta$
\[
\bigg| \sum_{n\in I} a_q(n) - \sum_{n\in I} a_q(n+\Delta) \bigg| \leq 2\Delta.
\]
Moreover $a_q(n+ir) =a_r(n) a_Q(n+ir)$ since $a_r(n)$ is periodic mod $r$, so that 
\begin{align*}
M \sum_{n\in I} a_q(n) & = \sum_{i=1}^M  \sum_{n\in I} a_q(n+ir) + \Theta((M^2+M)r)\\
&=  \sum_{n\in I}  a_r(n)  \sum_{i=1}^M   a_Q(n+ir) + \Theta((M^2+M)r),
\end{align*}
where $\Theta(t)$ means $\leq |t|$.
Therefore
\[
|\mathbb E_q(a)| \leq \frac 1  {MN} \sum_{n\in I} \bigg| \sum_{i=1}^M   a_Q(n+ir)\bigg| +  \frac{(M+1)r}N .
\]
Squaring and Cauchying gives
\[
|\mathbb E_q(a)|^2 \leq \frac 2{M^2}\sum_{1\leq i,j\leq M} \frac 1N \sum_{n\in I}   a_Q(n+ir) \overline{a_Q(n+jr)} +  \frac{2(M+1)^2r^2}{N^2}.
\]
The $i=j$ terms each contribute $\leq 1$ and so 
\[
|\mathbb E_q(a)|^2 \leq  \frac{2(M+1)^2r^2}{N^2} + \frac 2M +  \frac 2{M^2}\sum_{1\leq i\ne j\leq M} \mathbb E_Q(A_{(ir,jr)}),
\]
  so that \eqref{BasicInequality} holds.
\end{proof}

\subsection{Iterating the first bound}

\begin{prop} \label{prop: First thm}
We will write $q=q_1q_2\cdots q_kQ$ where the $q_i$ and $Q$ are pairwise coprime, and
   let $M_j=\lfloor(N/q_j)^{2/3}\rfloor$ for each $j$.  Then
  \[
|\mathbb E_q(a)| \leq 4 \sum_{i=1}^k M_i^{-2^{-i}}  
+ 4   \textbf{E}_{Q}  (a_{(M_1,\cdots, M_k) } )^{1/2^k}.
\]
\end{prop}

This is very similar to \cite[Lemma 3.1]{GR} though we include the argument here for completeness.

\begin{proof} [Proof of Proposition \ref{prop: First thm}]
Let $Q_j=q_{j+1}\cdots q_k\cdot Q$ so that $Q_k=Q$ and $q=q_1\cdots q_jQ_j$.

Applying Lemma \ref{lem: BasicIneq} to $q=Q_0=q_1Q_1$ we obtain
\[
|\mathbb E_q(a)|^2\leq \frac 5{M_1}  + 2 \textbf{E}_{Q_1}^{(1)} (a_{(M_1) } )  .
\]
 We Cauchy to obtain
 \[
|\mathbb E_q(a)|^4\leq \frac {2\cdot 5^2}{M_1^2} 
+ 8 \textbf{E}_{Q_1}^{(2)} (a_{(M_1) } )  .
\]
We now apply  Lemma \ref{lem: BasicIneq} to $Q_1=q_2Q_2$, so that
\[
\bigg|\mathbb E_{Q_1}\bigg(a_{\big[\substack{h_{1,0}    \\ h_{1,1} }\big]} \bigg)\bigg|^2
\leq \frac 5{M_2} 
+  \frac 2{M_2^2} \sum_{1\leq h_{2,0} \ne h_{2,1} \leq M_2}  \bigg|\mathbb E_{Q_2}\bigg(a_{\big[\substack{h_{1,0}, h_{2,0}    \\ h_{1,1}, h_{2,1} }\big]} \bigg)\bigg|  ,
\]
so that
 \[
|\mathbb E_q(a)|^4\leq \frac {2\cdot 5^2}{M_1^2} +  \frac {2^3\cdot 5}{M_2}
+ 2^4  \textbf{E}_{Q_2}^{(1)} (a_{(M_1,M_2) } ).
\]
We now prove that for each $\ell\ge 1$ we have, for $c_\ell=\prod_{j=2}^\ell j^{2^{\ell-j}}=\ell c_{\ell-1}^2$
with $g_\ell=5^{2^{\ell-1}}c_\ell, b_\ell = 2^{2^\ell-1}c_\ell$,
\[
|\mathbb E_q(a)|^{2^\ell}\leq g_\ell   \sum_{i=1}^\ell M_i^{-2^{\ell-i}}  
+ b_\ell   \textbf{E}_{Q_\ell}^{(1)} (a_{(M_1,\cdots, M_\ell) } )
\]
by induction.
Now this holds for $\ell=1$. Cauchying gives
\[
|\mathbb E_q(a)|^{2^{\ell+1}}\leq   g_{\ell+1} \sum_{i=1}^\ell M_i^{-2^{\ell+1-i}}  
+ b_{\ell+1}  \textbf{E}_{Q_\ell}^{(2)} (a_{(M_1,\cdots, M_\ell) } ) .
\]
We now apply  Lemma \ref{lem: BasicIneq} to $Q_\ell=q_{\ell+1}Q_{\ell+1}$, so that
\[
\bigg| E_{Q_\ell}\bigg(  a_{\big[\substack{h_{1,0},\dots, h_{\ell,0}  \\ h_{1,1}, \dots, h_{\ell,1}}\big]}\bigg) \bigg|^2
\leq \frac 5{M_{\ell+1}} 
+  \frac 2{M_{\ell+1}^2} \sum_{1\leq h_{{\ell+1},0} \ne h_{{\ell+1},1} \leq M_{\ell+1}}  \bigg|\mathbb E_{Q_{\ell+1}}\bigg(f_{\big[\substack{h_{1,0},\dots, h_{\ell+1,0}  \\ h_{1,1}, \dots, h_{\ell+1,1}}\big]} \bigg)\bigg| 
\]
and therefore, as $5b_{\ell+1}/2\leq g_{\ell+1} $
\[
|\mathbb E_q(a)|^{2^{\ell+1}}\leq g_{\ell+1}  \sum_{i=1}^{\ell+1} M_i^{-2^{\ell+1-i}}  
+ b_{\ell+1}    \textbf{E}_{Q_{\ell+1}}^{(2)} (a_{(M_1,\cdots, M_{\ell+1}) } ).
\]
This proves the induction. Taking $\ell=k$, and then $2^k$th roots we have
\[
|\mathbb E_q(a)| \leq g_k^{1/2^k}   \sum_{i=1}^k M_i^{-2^{-i}}  
+ b_k^{1/2^k}   \textbf{E}_{Q}^{(1)} (a_{(M_1,\cdots, M_k) } )^{1/2^k}.
\]
Now
$\log c_k^{1/2^k} = \sum_{j=2}^k  2^{-j} \log j \leq \kappa:=  \sum_{j\geq 2}  2^{-j} \log j$, and so
$b_k^{1/2^k}< g_k^{1/2^k} < \sqrt{5} e^\kappa =3.715647213\dots <4$, so the result follows.
\end{proof}

\subsection{Sums in an interval}

\begin{proof} [Proof of Lemma \ref{Lem: BdsInInterval}]
Let $N_q$ be the least residue of $N \pmod q$ so that $0\leq N_q<q$.
Then $I$ may be partitioned into $\lfloor \frac Nq\rfloor$ intervals of length $q$, each yielding a complete sum by periodicity, and one left over interval $I_q$ of length $N_q$, so that 
\[
\sum_{n\in I} a_q(n)    = \sum_{n\in I_q} a_q(n)   + \bigg\lfloor \frac Nq\bigg\rfloor \sum_{r \pmod q} a_q(r)   
\]
and therefore
 \[
\frac 1N \bigg|  \sum_{n\in I} a_q(n) \bigg|  \leq    \bigg|  \frac{1}q \sum_{r \pmod q} a_q(r)     \bigg|  +\frac 1N \bigg|  \sum_{n\in I_q} a_q(n) \bigg| .
\]

Now if $q=\prod_p p^{e_p}$ then $a_q=\prod_p a_{p^{e_p}}$ and so, as we saw above,
 \[
 \bigg|  \frac{1}q \sum_{r \pmod q} a_q(r)     \bigg|  =    \prod_{ p^e\| q} \bigg|  \frac{1}{p^e} \sum_{r \pmod {p^e}} a_{p^e}(r)     \bigg| .
\]
For the incomplete sum we use  Fourier analysis:
\begin{align*}
\sum_{n\in I_q} a_q(n) &= \sum_{r \pmod q}  a_q(r) \sum_{n\in I_q} \frac 1q \sum_{b \pmod q}  e\bigg( \frac{(n-r)b}q\bigg)\\
 & = \frac 1q \sum_{b \pmod q}\bigg(  \sum_{r \pmod q} a_q(r)   e\bigg( \frac{-br}q\bigg)  \bigg)  \bigg( \sum_{n\in I_q}  e\bigg( \frac{bn}q\bigg) \bigg)\\
   &\ll \frac 1q \sum_{-\frac q2<b\leq \frac q2}\bigg|  \sum_{r \pmod q}a_q(r)  e\bigg( \frac{-br}q\bigg)  \bigg|
  \min \bigg\{ \frac q{|b|} , N_q\bigg\}   \\
  &\ll   \log (N_q+2) \cdot \max_{b \pmod q} \bigg|  \sum_{r \pmod q}a_q(r)   e\bigg( \frac{br}q\bigg)  \bigg|.
\end{align*}
Therefore using the Chinese Remainder Theorem we obtain
\[
\frac 1N \bigg|  \sum_{n\in I_q} a_q(n) \bigg|   \ll \frac {q \log q}N \cdot \  \prod_{ p^e\| q} \ \max_{b \pmod {p^e} }  \bigg|  \frac{1}{p^e} \sum_{r \pmod {p^e}} a_{p^e}(r)     
e\bigg( \frac{br}{p^e}\bigg) \bigg| .
\]
Combining these two bounds we obtain \eqref{eq: CRT}.
\end{proof}



\section{Putting it all together} \label{sec: Together}

 If $q=Qr$ with $(Q,r)=1$ then there exist integers $a,b$ for which $aQ+br=1$.
We can write $\chi_q=\chi_Q\chi_r$ where $\chi_Q=\chi_q^{br}$ is a character mod $Q$ and $\chi_r=\chi_q^{aQ}$ is a character mod $r$ and they are both primitive characters if $\chi_q$ is. Moreover
$e(\frac mq) = e(\frac {ma }r)e(\frac {mb}Q) $, so that $a_q=a_Qa_r$ where 
$a_Q(n)=\chi_Q(f(n)) e(\frac{bg(n)}Q)$ and $a_r(n)=\chi_r(f(n)) e(\frac{ag(n)}r)$.

\begin{theorem} \label{thm: Almost there}
Suppose that $f(x),g(x)\in \mathbb Q(x)$ and $f(x)$ and $g(x)$ are not both constants.
Given $\epsilon>0$ and integer $k\geq 1$, select $\eta=\epsilon/2^{k+3}D$  with $D= 2^{k+1}(\deg f+\deg g)$.
Suppose $N$ and $q$ satisfy  $q=q_1q_2\cdots q_kQ$ where the $q_i$ and $Q$ are pairwise coprime,  
each $q_j\leq N^{1-\epsilon}$,  $(Q,\Delta(f,g,k))=1$ and $k<N^\eta$.
 Let $\chi$ be a character mod $q$, and define $\chi_Q$ as above.
 Let $Q^*=Q$ unless $g$ is a polynomial of degree $\leq k+1$ in which case
$\chi_Q^{r_f}$ is induced from a primitive character of conductor $Q^*$.
  If $Q<N$ or if   $Q$ is prime and $N\leq Q<N^{2-\epsilon}$,   then 
 \[
\bigg| \frac 1N \sum_{n\in I} \chi(f(n)) e\bigg( \frac{ g(n)}q  \bigg)  \bigg| \ll    (Q^*)^{-\eta},
\]
provided $k\cdot 2^k\ll \log\log Q^*$,  and $Q^*\geq (\log Q)^{12D}$.
\end{theorem}

To establish Theorem \ref{thm: main} we will need to factor $q=q_1q_2\cdots q_kQ$ with $Q$ as in the hypothesis of 
  Theorem \ref{thm: Almost there} so as to apply this result.
We begin this section with a technical lemma.

 
\begin{lemma} \label{lem: counts} For any positive integers $d$ and $M$
\[
 \sum_{\substack{1\leq h_{0} \ne h_{1} \leq M \\ d| h_0-h_1 }} 1 < \frac{M^2}d.
\]
\end{lemma}

\begin{proof}
Now if $M= ud+v$ where $0\leq v<d$ then 
 $\#\{ 1\leq h\leq M: h\equiv b \pmod {d}\} = u+1$ for $1\leq b\leq v$ and $=u$ otherwise,
 so that 
 \[
  \frac 1{M^2} \sum_{\substack{1\leq h_{0} \ne h_{1} \leq M   \\ d |  h_0-h_1}}  1
  =   \frac{v\cdot u(u+1) + (d-v) \cdot u(u-1)}{M^2} =
  \frac 1{d} - \frac{v^2+ud ^2}{d(ud+v)^2} < \frac 1{d} .\qedhere
 \]
\end{proof}

\begin{proof} [Proof of Theorem \ref{thm: Almost there}]
We begin by applying Proposition \ref{prop: First thm} with 
\[
a_q(n)= \chi_q(f(n)) e\bigg(\frac{g(n)}q\bigg)
\]
 where $\chi_q$ is a given primitive character mod $q$.

We have
\[
(a_q)_{\big[\substack{h_{1,0},\dots, h_{k,0}  \\ h_{1,1},\dots, h_{k,1}}\big]}(n) = 
 \chi_q(f^*(n)) e\bigg(\frac{g^+(n)}q\bigg)
\]
 where $f^*=f^*_{\big[\substack{h_{1,0}, \dots, h_{k,0}  \\ h_{1,1} ,\dots, h_{k,1}}\big]}$ and $g^+=g^+_{\big[\substack{h_{1,0}, \dots, h_{k,0}  \\ h_{1,1} ,\dots, h_{k,1}}\big]}$. Therefore the  $\textbf{E}_{Q}  (a_{(M_1,\cdots, M_k) } )$ of Corollary \ref{cor: First thm} is
\[
\frac 1{M_1^2\cdots M_k^2} \sum_{\substack{1\leq h_{j,0} \ne h_{j,1} \leq M_j \\ \text{for } 1\leq j\leq k}}
\bigg| \frac 1N \sum_{n\in I}   \chi_Q\bigg(f^*_{\big[\substack{h_{1,0},\dots, h_{k,0}  \\ h_{1,1},\dots, h_{k,1}}\big]}(n)\bigg) e\Bigg(\frac{g^+_{\big[\substack{h_{1,0},\dots, h_{k,0}  \\ h_{1,1},\dots, h_{k,1}}\big]}(n)}Q\Bigg)\bigg|
\]
 and then by Lemma \ref{Lem: BdsInInterval} this is
 \begin{equation} \label{eq: ComboBound}
\ll \bigg( 1+ \frac {Q \log Q}N\bigg)  \frac 1{M_1^2\cdots M_k^2} \sum_{\substack{1\leq h_{j,0} \ne h_{j,1} \leq M_j \\ \text{for } 1\leq j\leq k}}  \prod_{ p^e\| Q} \ \max_{b \text{ mod }{p^e} }  \bigg|  \frac{1}{p^e} \sum_{r \pmod {p^e}} 
  \chi_{p^e} (f^*(n)) e\bigg(\frac{g^+(n)+bn}{p^e} \bigg) \bigg| .
 \end{equation}
Now if $g$ is a polynomial of degree $\leq k+1$  and $\chi_{p^e}^{r_f}$ is principal then we will only apply the trivial bound $\leq 1$ for the 
$p^e$-th term in the product over $p^e\| Q$ in \eqref{eq: ComboBound}.

Suppose that   $N\leq Q\leq N^{2-\epsilon}$ with $Q=p$ prime and assume that if $g$ is a polynomial of degree $\leq k+1$ then $\chi_{p}^{r_f}$ is not principal. Then \eqref{eq: ComboBound} is, by the first part of Proposition \ref{prop: vp rev}, 
 \[
\ll   \frac {Q \log Q}N   \frac 1{M_1^2\cdots M_k^2} \sum_{\substack{1\leq h_{j,0} \ne h_{j,1} \leq M_j \\ \text{for } 1\leq j\leq k}}   2^{k+1}DQ^{-\frac{\max\{ 0,1-v_p(\mathbf{h})\} }2},
 \]
 and $Q^{- \max\{ 0,1-v_p(\mathbf{h})\} }= Q^{-1}Q^{ \min\{ 1,v_p(\mathbf{h})\} }
 \leq Q^{-1}  \prod_{j=1}^k (h_{j,0} - h_{j,1},p)$. Therefore \eqref{eq: ComboBound} is, by Lemma \ref{lem: counts},
 \[
\ll   \frac {Q^{\frac 12+o(1)} }N  \prod_{j=1}^k  \frac 1{M_j^2} \sum_{\substack{1\leq h_{j,0} \ne h_{j,1} \leq M_j \\ \text{for } 1\leq j\leq k}}   (h_{j,0} - h_{j,1},p)^{\frac{1}2} \leq \frac {Q^{\frac 12+o(1)} }N \bigg( 1 +\frac 1{p^{1/2}}\bigg)^k \leq \frac {Q^{\frac 12+o(1)} }N \ll Q^{-\epsilon/2}.
 \]
 By Corollary \ref{cor: First thm} we deduce that
 \begin{equation} \label{eq: BoundWrittenOut}
\bigg| \frac 1N \sum_{n\in I} \chi(f(n)) e\bigg( \frac{ g(n)}q  \bigg)  \bigg| \ll    kN^{-\frac{\epsilon}{2^k}}+\left( Q^{-\epsilon/2}\right)^{\frac{1}{2^k}}\ll   Q^{-\eta}
 \end{equation}
 since $ kN^{-\frac{\epsilon}{2^k}}< N^{\eta-\frac{\epsilon}{2^k}}< N^{-2\eta}< Q^{-\eta}$ as $Q<N^2$.

Now suppose that $Q\leq N$ and $p^{e_p}\|Q$. Then $Q^*=\prod_{p|Q, s_p\geq 1} p^{s_p}$ where
 $s_p=e_p$ unless $g$ is a polynomial of degree $\leq k+1$  and $\chi_{p^{e_p}}^{r_f}$ is induced from a primitive character of conductor $p^{\ell_p}$   in which case  $s_p=\ell_p$.  
 The  second estimate of  Proposition \ref{prop: vp rev} implies that \eqref{eq: ComboBound} is 
\[
\ll \bigg( 1+ \frac {Q \log Q}N\bigg)  \frac 1{M_1^2\cdots M_k^2} \sum_{\substack{1\leq h_{j,0} \ne h_{j,1} \leq M_j \\ \text{for } 1\leq j\leq k  }}  
\prod_{p^{s_p}\| Q^* } D\, p^{- \frac { \max\{ 0,s_p-v_p(\mathbf{h)\} }}{4D}  } ,
\]    
 Now $\prod_{p\| Q^*}  D=D^{\omega(Q^*)} $. Writing $\Pi(\mathbf{h}):= \prod_{i=1}^k ( h_{i,1}-h_{i,0})$, we have
\[
 ( Q^*, \Pi(\mathbf{h} )) =  \prod_{ p^{s_p}\| Q^*  }  p^{\min\{ s_p,v_p(\mathbf{h)\} }}
= Q^*  \prod_{ p^{s_p}\| Q^*  }  p^{-\max\{ 0,s_p-v_p(\mathbf{h)\} } },
\]
and so  \eqref{eq: ComboBound} is
 \[
\ll D^{\omega(Q^*)}  \bigg( 1+ \frac {Q\log Q}N\bigg) (Q^*)^{-\frac 1{4D}} \frac 1{M_1^2\cdots M_k^2} \sum_{\substack{1\leq h_{j,0} \ne h_{j,1} \leq M_j \\ \text{for } 1\leq j\leq k  }}   ( Q^*,  \Pi(\mathbf{h}) )^{\frac 1{4D}} .
\]   
Now if $( Q^*,  \Pi(\mathbf{h}) )=u$   then $u|Q^*, u|\Pi(\mathbf{h}) $ and so
\[
\frac 1{M_1^2\cdots M_k^2} \sum_{\substack{1\leq h_{j,0} \ne h_{j,1} \leq M_j \\ \text{for } 1\leq j\leq k  }}   ( Q^*,  \Pi(\mathbf{h}) )^{\frac 1{4D}}  \leq \sum_{u|Q^*} u^{\frac 1{4D}}
 \frac 1{M_1^2\cdots M_k^2} \sum_{\substack{1\leq h_{j,0} \ne h_{j,1} \leq M_j \\ \text{for } 1\leq j\leq k\\ u| \Pi(\mathbf{h})  }} 1.
\]
This last sum is
\[
\sum_{d_1\cdots d_k=u}  \prod_{j=1}^k \frac 1{M_j^2} \sum_{\substack{1\leq h_{j,0} \ne h_{j,1} \leq M_j \\ d_j|  h_{j,1}-h_{j,0}  }} 1
< \sum_{d_1\cdots d_k=u} \frac 1{d_1\cdots d_k}= \frac{\tau_k(u)}{u}.
\]
by Lemma \ref{lem: counts}, and 
\[
\sum_{u|Q^*} u^{\frac 1{4D}}  \frac{\tau_k(u)}{u} \leq \prod_{p|Q^*}  \bigg( 1 - \frac 1{p^{1-\frac 1{4D} }}\bigg)^{-k}  <
3^{k\, \omega(Q^*)} 
\]
so 
 \[
\textbf{E}_{Q}  (a_{(M_1,\cdots, M_k) } ) \ll (3^kD)^{ \omega(Q^*)}   \bigg( 1+ \frac {Q\log Q}N\bigg) (Q^*)^{-\frac {1}{4D}} .
\] 
Now $Q<N$ and  $\log Q \ll (Q^*)^{\frac {1}{12D}} $ by the hypothesis so that $1+ \frac {Q\log Q}N\ll  (Q^*)^{\frac {1}{12D}} $.
Also  $ (3^kD)^{ \omega(Q^*)} \leq (Q^*)^{  c\frac{\log 3^kD}{\log\log Q^*}}   \ll (Q^*)^{\frac {1}{12D}} $ 
since  $ \omega(Q^*)\ll \frac{\log Q^*}{\log\log Q^*}$  and as $k\cdot 2^k\ll \log\log Q^*$ by the hypothesis. Therefore
\[
\textbf{E}_{Q}  (a_{(M_1,\cdots, M_k) } ) \ll  (Q^*)^{-\frac {1}{12D}} .
\] 

By Corollary \ref{cor: First thm} we deduce that
 \begin{equation} \label{eq: BoundWrittenOut}
\bigg| \frac 1N \sum_{n\in I} \chi(f(n)) e\bigg( \frac{ g(n)}q  \bigg)  \bigg| \ll    kN^{-\frac{\epsilon}{2^k}}+(Q^*)^{-\frac {1}{2^{k}\cdot 12 D}} \ll (Q^*)^{-\eta}
 \end{equation}
 since $ kN^{-\frac{\epsilon}{2^k}}< N^{\eta-\frac{\epsilon}{2^k}}< N^{-\eta}< (Q^*)^{-\eta}$ as $Q^*<Q<N$. 
   \end{proof}

With this established we can prove our main theorem:

\begin{theorem} \label{thm: mainY}  Fix $\delta,\epsilon>0$. Suppose that $f(x), g(x)\in \mathbb Q(x)$ where $f(x)$ and $g(x)$ are not both constants.
There exists a constant $\eta=\eta(f,g,\delta,\epsilon)>0$ and an explicitly determinable  positive integer $\mathcal M=\mathcal M(f,g,\delta)$ such that if\\
 {\rm (a)} For any character $\chi$ mod $q$ where $q\in \mathcal N(y)$   and $y:=q^\delta$;\\
 and {\rm (b)} of Theorem \ref{thm: main} holds, 
then
 \begin{equation} \label{Key result M}
  \sum_{n\in I} \chi(f(n)) e\bigg( \frac{ g(n)}q  \bigg) \ll N / q_{\mathcal M}^{\eta}
 \end{equation}  
holds unless  $g(x)$ is  a polynomial of degree $\leq 2/\delta+1$ and $\chi^{r_f}$ is induced from a primitive character of conductor $q'$, in which case 
 \begin{equation} \label{Key result2 M}
  \sum_{n\in I} \chi(f(n)) e\bigg( \frac{ g(n)}q  \bigg) \ll N / (q'_{\mathcal M})^{\eta}
 \end{equation}  
 holds.  The implicit constants in \eqref{Key result M} and \eqref{Key result2 M} depend only on $f, g, \delta$ and $\epsilon$  (like $\eta$).
\end{theorem}

We will use classical Weyl-type arguments in the next section in the case that $g(x)$ is a polynomial of degree $\geq 1/\delta$ to extend the result in the second case, in order to obtain Theorem \ref{thm: main}.

 \begin{proof} [Proof of  Theorem \ref{thm: mainY}] An integer $k$ will be determined in the proof  and will be bounded in terms of $\delta$.
 If $g$ is a polynomial of degree $\leq k+1$ then suppose that 
$\chi^{r_f}$ is induced from a primitive character of conductor $q'$; otherwise let $q'=q$.
 We let $\mathcal M= \Delta(f,g,2/\delta)$; we will write $q=Qq_1q_2\cdots q_k$ with $k<2/\delta$ so that 
$\mathcal M$ is divisible by all of  the prime divisors of $\Delta(f,g,k)$ which will allow  us to apply Theorem \ref{thm: Almost there}. 

The value of $Q$ here will now be defined in three separate cases so that $(Q,\mathcal M)=1$ and $Q >(q'_{\mathcal M})^{\delta/2}$.
We will therefore be able to apply Theorem  \ref{thm: Almost there} with this value of $Q$, and we observe that $Q^*=Q$ in Theorem  \ref{thm: Almost there} because of how we constructed $q'$.

I)\ If $q'$ has a prime factor $p>y$ then let $Q=p$ (and this prime $p$ does not divide $\mathcal M$ as we can assume that $y$ is arbitrarily large, so that $Q_{\mathcal M}=Q$).    Here $Q>y=q^{\delta}\geq (q'_{\mathcal M})^{\delta}$.

II)\ If $q'_{\mathcal M}\leq y$ then let $Q=q'_{\mathcal M}$.

III) \ Otherwise write $q'_{\mathcal M}=\prod_p p^{e_p}$ where each $p^{e_p}\leq y$, and so $q'_{\mathcal M}$ has at least two prime power factors as $q'_{\mathcal M}>y$. 
Order the distinct prime power factors of $q'_{\mathcal M}$ as $p_1^{e_1} > p_2^{e_2}>  \dots$
We then let  $Q=p_1^{e_1} \cdots p_\ell^{e_\ell} $ where $p_\ell$ is chosen  maximally so that $Q\leq y$.
We claim that $y^{1/2}<Q\leq y$:   If $\ell=1$ then $p_1^{2e_1}>p_1^{e_1}p_2^{e_2} >y$ so that 
$Q=p_1^{e_1}>y^{1/2}$. If $\ell>1$ then $p_{\ell+1}^{e_{\ell+1}} < p_2^{e_2}<y^{1/2}$, else
$Q\geq p_1^{e_1}p_2^{e_2}>p_2^{2e_2}  >y$, and so
$Q=p_1^{e_1} \cdots p_{\ell+1}^{e_{\ell+1}}/p_{\ell+1}^{e_{\ell+1}}>y/y^{1/2}=y^{1/2}=q^{\delta/2}\geq (q'_{\mathcal M})^{\delta/2}$.

We now put  the prime power divisors of $q/Q$ in descending size order  and select $q_1$ to be the largest product of the first
few such prime power factors that is $\leq y$. We then select $q_2$ from the remaining $p^{e_p}$ in the same way and eventually
$q/Q=q_1q_2\cdots q_k$ where the $q_i$ and $Q$ are pairwise coprime, and
$q_1,\dots,q_{k-1}\in (y^{1/2},y]$ with $q_k\leq y$. 

Our inequalities for $Q$ and the $q_i$'s yield that
 $y^{k/2} <q'\leq y^{k+2}$ and so $k/2<1/\delta\leq k+2$, which implies that $1/\delta-2\leq k<2/\delta$.
 
 Now, since $Q=Q^*$, if $k\cdot 2^k\ll \log\log Q$ then 
  \[
\bigg| \frac 1N \sum_{n\in I} \chi(f(n)) e\bigg( \frac{ g(n)}q  \bigg)  \bigg| \ll    Q^{-\eta}\ll (q'_{\mathcal M})^{-\delta\eta/2}
\]
by  Theorem \ref{thm: Almost there}. Then \eqref{Key result M} and  \eqref{Key result2 M} follow by adjusting the value of $\eta$.
Now $k$ is fixed in the statement of the theorem and so if the inequality $k\cdot 2^k\ll \log\log Q$ fails then $Q$ is bounded, and then
\eqref{Key result M} and  \eqref{Key result2 M} follow trivially by adjusting the implicit constant.
 \end{proof}
  
 
  \section{Classical exponential sum}  \label{sec: Classics} 
  
 We use the following  the classical Weyl Sum estimate; see eg. Davenport \cite[Lemma 3.1]{davenport} or Montgomery \cite[Theorem 2]{montgomery}.

\begin{lemma} \label{weylsumlemma} Let $d\ge 2$ be an integer, and $\alpha_i \in \mathbb R$, $1 \le i \le d$. Suppose that 
$|\alpha_d - \frac{A}{Q}| \leq Q^{-2}$ for pairwise coprime integers $A$ and $Q>0$.  Then   for any $\varepsilon>0$ and  integer $N\geq 1$, 
\begin{equation}\label{weylsigma}
\Bigg| \sum_{n\leq N} e( \alpha_1n + \dots + \alpha_d n^d ) \Bigg| \ll_{\varepsilon,d}
 N^{1+\varepsilon} \left( \frac{1}{Q} + \frac{1}{N} + \frac{Q}{N^d} \right)^{\sigma}.
\end{equation}
We can take  $\sigma=\frac 1{d(d-1)}$ using the improvement of  Bourgain, Demeter and Guth \cite{bdg}  to Vinogradov's mean value theorem.
\end{lemma}
\begin{prop} \label{prop: Classic} Fix $\epsilon>0$, integer $d\geq 2$ and polynomial $g(x)\in \mathbb Z[x]$ of degree $d$.   There exists $\eta>0$, depending only on $\epsilon, g$ and $d$ such that for  any interval $I$ of length $N\ge q^{1/d+\epsilon}$  we have
\[
\bigg| \sum_{n\in I}   e\bigg( \frac{ g(n)}q  \bigg)  \bigg| \ll N/q^\eta.
\]
Here the $\ll$'s may depend   on $\epsilon, d$ and $g$.
\end{prop}

This is essentially ``best-possible'' (at least as formulated) since if we take $g(x)=x^d$ and $N=cq^{1/d}$ for some small $c>0$ then 
each $e ( \frac{ n^d}q  ) =1+O(c^d)$  if $n\leq N$ so that $\text{Re}(e ( \frac{ n^d}q  ))\gg 1$ and therefore
$\text{Re}( \sum_{n\leq N} e ( \frac{ n^d}q  ))\gg N$.

\begin{proof} 
We may assume that $\epsilon< 1/d^2$. 
Given $I$ we dissect $I$ into intervals $(M,M+N_0]$ of length $N_0:=\lfloor q^{1/d+\epsilon/2} \rfloor$ where $M$ is an integer, and use the trivial upper bound for the one part-interval. Provided $\epsilon\geq 2\eta$ we can obtain the result on each such sub-interval of length $N_0$ and add the subsequent upper bounds together to get the desired bound on $I$. Therefore henceforth we can take $N=N_0:=\lfloor q^{1/d+\epsilon/2} \rfloor$. In particular,
\begin{equation} \label{weyl3}
\frac 1q<\frac 1N<\frac q{N^d}\asymp q^{-\epsilon d/2}.
\end{equation}

For $g(x)=a_0+a_1x+\cdots+a_dx^d$ let $\alpha_1x + \dots + \alpha_d x^d=\frac{g(M+x)-g(M)}q$ so that $\alpha_d=\frac{a_d}q$, and we can write   $\frac{a_d}q=\frac AQ$ with $(A,Q)=1$ where $q\geq Q\geq q/a_d$ so that $Q\asymp q$. Therefore Lemma \ref{weylsumlemma} 
(with $\epsilon$ replaced by $\epsilon/2$) and \eqref{weyl3}  give
\[
\bigg| \sum_{M<n\leq M+N}   e\bigg( \frac{ g(n)}q  \bigg)  \bigg| \ll_{\varepsilon,d,g}
 N^{1+\varepsilon/2} \left( \frac{1}{q} + \frac{1}{N} + \frac{q}{N^d} \right)^{\sigma}
 \asymp  N\cdot \bigg( N^{1/2} q^{-1/2(d-1)}\bigg)^\epsilon\ll N/q^\eta
\]
provided $0<\eta < \epsilon/4d(d-1)$.
\end{proof}

\begin{theorem} \label{thm: mainZ}  Fix $\delta,\epsilon>0$. Suppose that $f(x)\in \mathbb Q(x)$ and $g(x)$ is  a polynomial of degree $D \geq 1/\delta$.
There exists a constant $\eta=\eta(f,g,\delta,\epsilon)>0$ and an explicitly determinable  positive integer $\mathcal M=\mathcal M(f,g,\delta)$ such that if \\
 {\rm (a)} For any character $\chi$ mod $q$ where $q\in \mathcal N(y)$ with  $q_{\mathcal M}>q^{1-\epsilon/(4D)}$ and $y:=q^\delta$;\\
 and {\rm (b)} of Theorem \ref{thm: mainY} hold,    then \eqref{Key result M} holds.
\end{theorem}
 
\begin{proof} Suppose that $g(x)$ has degree $D\geq 1/\delta$ and that $N\geq q^{\delta(1+\epsilon)} \geq  q^{1/D +\epsilon/D}$. Say $q_{\mathcal M}=q^{1-\lambda}$ with $\lambda<\epsilon/(4D)$.
 We can assume that $\chi^{r_f}$ is induced from a primitive character of conductor $q'$ where  $q'_{\mathcal M}<q_{\mathcal M}^\lambda$ else the desired result follows from \eqref{Key result M} or  \eqref{Key result2 M}  given by Theorem \ref{thm: mainY}  (adjusting the value of $\eta$).
Note that
\begin{equation} \label{qprimeub}
q'=(q'/q'_{\mathcal M})q'_{\mathcal M}\ll (q/q_{\mathcal M})q'_{\mathcal M}\leq q^{\lambda}q_{\mathcal M}^{\lambda} = q^{2\lambda-\lambda^2}.
\end{equation}

By definition we can write $f(n)=cF(n)^{r_f}$ so that $\chi(f(n))=\chi(c) \chi^{r_f}(F(n)) = \chi(c) \psi(F(n)) 1_{(\cdot,r)=1}(F(n))$ where $\psi$ has conductor $q'$ and $r=\prod_{p|q,\ p\nmid q'}p$. If $\chi(c)= 0$ or if $p|F$ for some prime $p$ dividing $q$, the  result follows trivially; otherwise our exponential sum in absolute value, divided by $N$, equals
\[
   \bigg|  \frac 1N \sum_{\substack{n\in I\\ (F(n),r)=1}}  \psi(F(n))e\bigg( \frac{ g(n)}q  \bigg) \bigg| 
  =   \frac 1N \bigg| \sum_{d|r} \mu(d)  \sum_{a \pmod {q'}} \psi(F(a)) \sum_{\substack{n\in I\\ d|F(n)\\ n\equiv a \pmod {q'}}}  e\bigg( \frac{ g(n)}q  \bigg) \bigg| 
  \]
  and 
 \[
  \sum_{\substack{n\in I\\ d|F(n)\\ n\equiv a \pmod {q'}}}  e\bigg( \frac{ g(n)}q  \bigg) =   \sum_{\substack{b \pmod d\\ F(b)\equiv 0 \pmod d}} \sum_{\substack{n\in I\\  n\equiv b \pmod d\\ n\equiv a \pmod {q'}}}  e\bigg( \frac{ g(n)}q  \bigg).
  \]
  The conditions in the last sum combine into one congruence $n\equiv C \pmod{dq'}$. The set of $\{ b \pmod d: F(b)\equiv 0 \pmod d\}$ can be determined by the Chinese Remainder Theorem so contains $\leq k^{\omega(d)}$ elements where $k=\deg F_+$  
  and $\omega(d)$ denotes the number of distinct prime factors of $d$. Taking absolute values everywhere and
  dividing by $N$ yields
  \[
   \bigg|  \frac 1N \sum_{\substack{n\in I\\ (F(n),r)=1}}  \psi(F(n))e\bigg( \frac{ g(n)}q  \bigg) \bigg| 
\leq 
 (k+1)^{\omega(r)} \max_{d|r}   \max_{C \mod {dq'}}   \bigg|  \frac 1{N/q' }\sum_{\substack{n\in I\\   n\equiv C \pmod {dq'}}}  e\bigg( \frac{ g(n)}q  \bigg) \bigg| 
\]
  Writing
$n=C+mdq'$ we get $g(n)=g(C)+dq'g_{C,d}(m)$ for some polynomial $g_{C,d}(x)$ of   degree $D$, with leading coefficient
$(dq')^{D-1}$ times the leading coefficient of $g$. So for some interval $J$ of length $N/dq'$ we have the upper bound
\[
\leq q^{o(1)} \max_{d|r}      \bigg|  \frac 1{N/q' }  \sum_{m\in J}  e\bigg( \frac{ g_{C,d}(m)}{q/dq'}\bigg) \bigg|, 
\]
which by \eqref{qprimeub} is trivially bounded by 
$$
q^{o(1)} \frac {q'}N\left(\frac {N}{dq'}+1\right) \le q^{o(1)}\left(\frac 1d+\frac {q^{2\lambda-\lambda^2}}{q^{1/D}}\right) \ll \frac 1{q_{\mathcal M}^{\eta}},
$$
for $d> q_{\mathcal M}^{2\eta}$.  Thus (adjusting $\eta$) we may assume that the maximum occurs at a value $d$ with $d<q^{\eta}$.
It follows from \eqref{qprimeub} that for $\eta<\lambda^2$ and $\lambda<\epsilon/(4D)$, we have
$$ 
N/dq' \ge q^{\frac 1D+\frac {\epsilon}D-2\lambda+\lambda^2-\eta}> q^{\frac 1D+\frac {\epsilon}D-2\lambda} >q^{\frac 1D+\frac {\epsilon}{2D}}>(q/dq')^{\frac 1D+\frac {\epsilon}{2D}}.
$$
 Therefore Proposition \ref{prop: Classic} applied with $\epsilon/(2D)$ in place of $\epsilon$ implies that  
\[
 \bigg|  \frac 1{N/q' }  \sum_{m\in J}  e\bigg( \frac{ g_{C,d}(m)}{q/dq'}\bigg) \bigg|  \ll \frac 1{d(q/dq')^\eta} \leq \frac 1{(q/q')^\eta}
 \leq \frac 1{(q_{\mathcal M}/q'_{\mathcal M})^\eta}
  \leq   \frac 1{q_{\mathcal M}^{(1-\epsilon) \eta}},
\]
 and the claim follows after adjusting the value of $\eta$.
\end{proof}

\begin{proof}  [Proof of  Theorem \ref{thm: main}] 
This follows immediately from combining   Theorems \ref{thm: mainY} and \ref{thm: mainZ}.
 \end{proof}

  \begin{remark} \label{rem: MCC}
 In the statement of Theorem \ref{thm: Almost there} we require that $\eta\asymp 4^{-k}$ and $k<N^\eta$, and in the proof of Theorem \ref{thm: mainY}  we found  that
$\delta\asymp 1/k$. Thus, for  $N=y^{1+\epsilon}$ and $y=q^\delta$, we need $\delta \gg \frac 1{\log\log q}$, and hence $N\geq q^{ \frac c{\log\log q}}.$ 
\end{remark}

  \section{Corollaries}  \label{sec: Corr} 
  
\begin{proof} [Sketch of the proof of Corollary \ref{cor: mainnonprincipal2}] Suppose that $\chi$ is induced from a primitive character $\chi_Q$ of conductor $Q>1$ so that $\chi(n) =  \chi_Q(n)  1_{(n,r)=1}$ where $r=\prod_{p|q,  p\nmid Q}p$. 
   Therefore
\[
 \sum_{n\in I} \chi(n) = \sum_{n\in I} \chi_Q(n) \sum_{\ell | n,r} \mu(\ell) = \sum_{\ell |r} \mu(\ell)  \chi_Q( \ell  ) \sum_{m\in \frac 1\ell I} \chi_Q( m)   ,
\]
where $n=\ell m$ and $\frac 1\ell I$ is the interval $(\frac x\ell , \frac{x+N}\ell ]$ if $I=(x,x+N]$, and this is 
\[
\leq \tau(r) \cdot \max_{\ell |r}  \bigg|  \sum_{m\in \frac 1\ell I} \chi_Q(m)  \bigg| 
\leq q^{o(1)} \max_{\ell |r}  \bigg|  \sum_{m\in \frac 1\ell I} \chi_Q(m)  \bigg| \ll Q^{1/2}q^{o(1)}
\]
in absolute value as $\tau(r)\leq  r^{o(1)} \leq  q^{o(1)}$, and then by the Polya-Vinogradov theorem.
Thus we may assume that $Q>N^{2-3\eta}$ else the above is  $\ll  N^{1-\eta}$.
Also we may assume that the maximum occurs with $\ell<N^{2\eta}$ else
the sum is trivially bounded by the length of the interval which is $\leq N/\ell+1\ll N^{1-\eta}$.
Therefore taking $f(x)=x, g(x)=0$ in   Theorem \ref{thm: main}, with 
$q$ and $N$ replaced by $Q$ and   $N/\ell$, respectively, we obtain the bound
$ \ll N / (\ell  Q^{\eta}) \leq  N/N^{2\eta -3\eta^2} \ll N^{1-\eta} $.
 \end{proof}

  \begin{proof}[Proof of Corollary \ref{cor: L values}]
Fix $\epsilon>0$ very small and  $\delta=k\epsilon$ for any given integer $k\geq 1$.
Consider those integers $q$, coprime to $\mathcal M$, for which
either $P_1$ is prime and $y:=q^\delta\geq Y:=\max\{ P_1^{1/2},P_2\}>q^{\delta-\epsilon}$,
or $P_1$ is a prime power and $y\geq Y:=P_1>q^{\delta-\epsilon}$, so that $q\in \mathcal N(y)$.
Applying Corollary \ref{cor: mainnonprincipal2}    for $I=(0,N]$  we have
\[
\sum_{y^{1+\epsilon}<n\leq q} \frac{\chi(n)}{n^{1+it}}  
= N^{-1-it}  \sum_{n\leq N} \chi(n)\bigg|_{y^{1+\epsilon}}^q +(1+it  )\int_{y^{1+\epsilon}}^q N^{-2-it} \sum_{n\leq N} \chi(n)\ dN
\]
\[
\ll y^{-\eta}(1 +(1+|t|) \eta^{-1})\ll 1
\]
as $\eta$ is given and $t = q^{o(1)}< y^\eta$.
For the sum over $n>q$ we use the Polya-Vinogradov theorem in the analogous argument and obtain $\ll (1+|t|) \log q/\sqrt{q}\ll 1$. Therefore 
\begin{align*}
|L(1+it,\chi)| &\ll \sum_{n\leq y^{1+\epsilon}}  \bigg|\frac{\chi(n)}{n^{1+it}}  \bigg| + O(1) \leq\sum_{n\leq y^{1+\epsilon}} \frac 1n+O(1)\\
&=(1+\epsilon ) \log y+O(1) =  \log Y +O( \epsilon \log q).
\end{align*}
The result follows letting $\epsilon\to 0$.
\end{proof}

  \begin{proof}[Proof of Corollary \ref{eq:BT-Titchmarsh}] Following the proof of \cite[Theorem 1.5]{XZ}, we may replace \cite[Lemma 3.6]{XZ} by Corollary \ref{cor: mainnonprincipal2}; this still yields the inequality \cite[Eq.(5.19)]{XZ}, and the remaining argument goes through unchanged.
   \end{proof}

  \bibliographystyle{plain}

\end{document}